%% file: main.tex
\documentclass{article}
\input{preamble}

\usepackage[margin=1.67in]{geometry}

\title{Bicycle tracks with  hyperbolic monodromy -- \\ results and conjectures}
\author{G. Bor,  L. Hern\'andez-Lamoneda, S. Tabachnikov}
\begin{document}
\maketitle

{\em Dedicated to the memory of Robert Foote  (1953-2024).\\ }
\begin{abstract} We find new necessary and sufficient conditions for the bicycling monodromy of a closed plane curve to be hyperbolic. Our main tool is the ``hyperbolic development" interpretation of the bicycling monodromy of plane curves. Based on computer experiments, we pose two conjectures concerning the bicycling monodromy of strictly convex closed plane curves.

\end{abstract}

\tableofcontents
\section{Introduction}

\paragraph{The bicycling equation and monodromy.} A simple  bicycling  model consists of the planar motion of a directed line segment, whose endpoints trace two curves, the rear  and front tracks, such that the segment is tangent at each moment to the rear track (the `no-skid' condition). See Fig.~\ref{fig:tracks} (left).

\begin{figure}[h]
\centering
\def\svgwidth{.8\textwidth}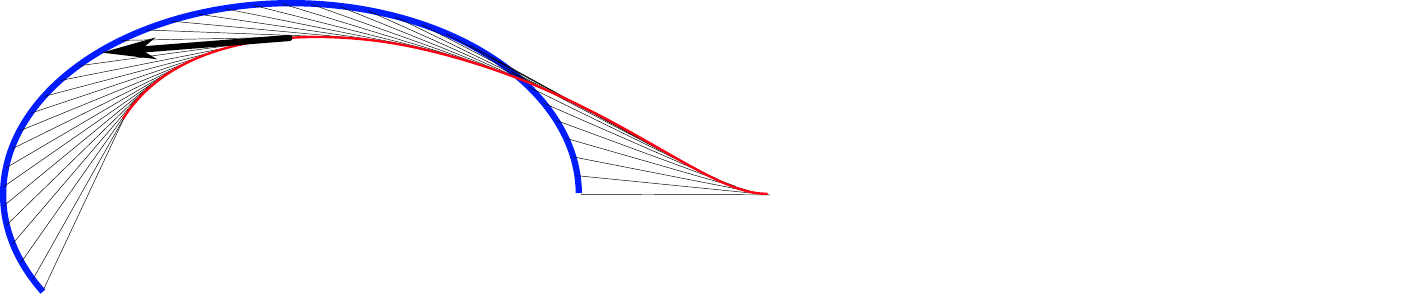
\caption{The front ($\G$) and back ($\g$) bicycle tracks and the circle map $b(\G)$ generated by $\G$.}\label{fig:tracks}
\vspace{-1em}
\end{figure}

Given an oriented  smoothly immersed front track $\G$ (not necessarily closed), the associated {\it bicycling transport} (with bike length 1 throughout this article, except the last section) is  a circle map 
associating to an initial orientation of the bike at the  initial point of $\G$ its  terminal orientation after riding along $\G$; we identify the circles of the  initial and the terminal positions of the bicycle by parallel translation.
 If one parametrizes the front track by $\G:[t_0,t_1]\to \R^2$, and  the back track by $\g(t)=\Gamma(t) + (\cos\theta(t), \sin\theta(t)),$ so that 
$|\G(t)-\g(t)|=1$ for all $t\in[a,b]$, then the no-skid condition, $\g'\| (\g-\G)$, is given by the {\em bicycling equation}: 
\be\label{eq:bei}
\theta'(t)=X'(t)\sin\theta(t)-Y'(t)\cos\theta(t), \quad \mbox{where } \G(t)=(X(t), Y(t)). 
\ee
The associated bicycling transport 
$$b(\G):S^1\to S^1$$
 is the circle map $e^{i\theta(t_0)}\mapsto e^{i\theta(t_1)},$ where $\theta(t)$ is a solution of \eqref{eq:bei}. See Fig.~\ref{fig:tracks} (right).

In fact, as R. Foote found in his 1989 pioneering article on the subject  \cite{F}, the bicycle equation \eqref{eq:bei}  arises from  a certain linear connection, 
associating to $\G(t)$   the  linear system 
\be\label{eq:foote}
\y'=A(t)\y, \ \mbox{ where } A(t)=-{1\over 2}\mat{X'&Y' \\ Y'&-X'} \mbox{ and } \y(t)\in \R^2.
\ee 

\setlength{\textfloatsep}{0pt} 

\begin{wrapfigure}{r}{0.18\textwidth}
\vspace{-1em}
\def\svgwidth{.18\textwidth}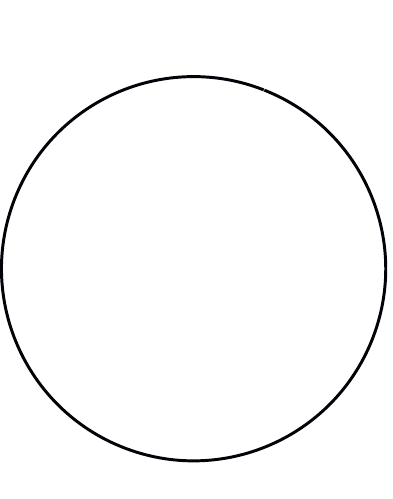\end{wrapfigure}
 The corresponding  parallel transport along $\G$ is the  linear map  $\R^2\to\R^2,$ mapping $\y(t_0)\mapsto \y(t_1)$, where $\y(t)$ is a solution to equation \eqref{eq:foote}. This linear map induces a circle map $S^1\to  S^1$, as follows. 
 
Identify the unit circle $S^1$ with the projective line $\RP^1$ via the stereographic projection, $e^{i\theta}\mapsto p=\tan(\theta/2)$. Then the parallel transport of \eqref{eq:foote} induces a projective map $\RP^1\to \RP^1$, i.e., a  {\em Möbius transformation}, $p\mapsto (ap+b)/(cp+d)$, an element  of $\rm{PSL}_2(\R)$. 

\begin{theorem}[R.~Foote \cite{F}]\label{thm:foote} The bicycle transport $b(\G)$ of equation \eqref{eq:bei} is a M\"obius transformation. In fact, it is the projectivized parallel transport of equation \eqref{eq:foote}.
\end{theorem}

See Section \ref{sec:ricatti} below for a proof. 

\smallskip

 For a {\em closed} front track,  $b(\G)$ is the {\em bicycle monodromy} (or holonomy) of $\G$, well-defined up to conjugation, depending on the initial point.   A non-trivial M\"obius  transformation is either hyperbolic, elliptic or parabolic, and is hyperbolic if and only if its action on $S^1$ has exactly two fixed points, one attracting  and one repelling. Accordingly, hyperbolic monodromy means that starting at any
 point along $\G$ there are exactly two initial orientations of the bicycle frame at this point that result in a closed rear track upon riding  once around $\G.$ Furthermore, for every other initial orientation, riding the bike along $\G$ many times, the rear  track converges    in forward time to the ``attracting'' closed back track, and to the ``repelling" one in backward time.  See Fig.~\ref{fig:hm}c.

\begin{figure}[h]
\centering
\def\svgwidth{1\textwidth}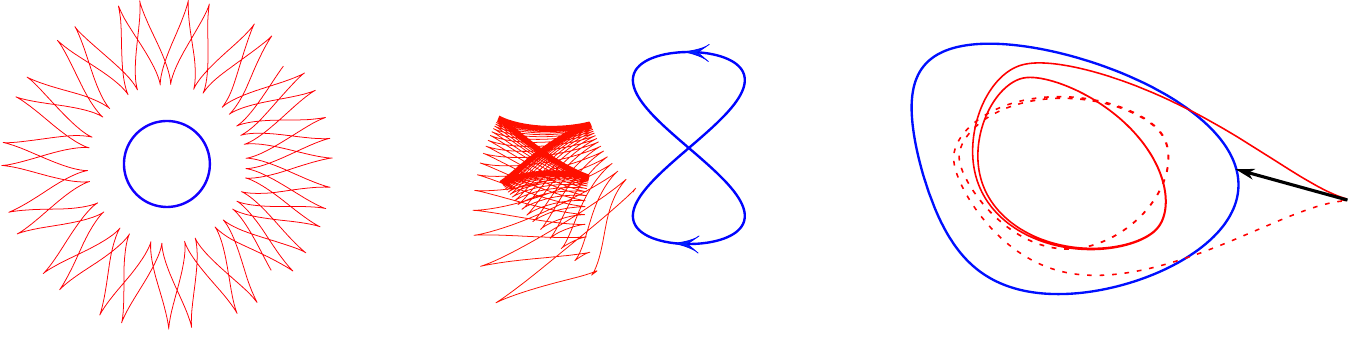
\caption{Bicycle monodromies: (a) elliptic; (b) and (c) hyperbolic. }\label{fig:hm}

\end{figure}

The {\em Menzin Conjecture}, dating from 1906 \cite{M}, states that {\em  $b(\G)$ is hyperbolic if $\G$ is a closed simple front track $\G$ enclosing an area   $A(\G)>\pi$}. Menzin studied planimeters, the mechanical devices to measure areas of plane domains. The mathematical model of one of them,  the hatchet planimeter, coincides with our bicycle model. 
Here is a quotation from \cite{M}:
\begin{quote}
\itshape
A curious property of the instrument was noticed while using it on very large areas. If the average line across the area is long in comparison with the length of arm, and the tracing point is dragged continuously around the contour, the tractrix will approach, asymptotically, a limiting closed curve. From purely empirical observations, it seems that this effect can be obtained so long as the length of arm does not exceed the radius of a circle of area equal to the area of the base curve.
\end{quote}
 This was proved in 2009  under the assumption that $\G$ is convex \cite{LT}. The non-convex case is still open, as far as we know. In Section \ref{sec:examples} below we show via some simple examples that the  condition $A(\G)>\pi$ is not necessary for hyperbolicity of $b(\G)$.

\paragraph{Statement of results.} 

We prove two related  results for closed $\G$. The first is another  sufficient condition on $\G$ implying   hyperbolicity of $b(\G)$, which applies to a non-convex or even non-simple $\G$. The second is a necessary condition for hyperbolicity  of $b(\G)$ for convex $\G$. 

\begin{theorem}\label{thm:suf}If an  immersed  closed smooth plane curve has  curvature  $|\kappa|\leq 1$ pointwise, but $|\kappa|$ is not identically 1,  then its bicycle monodromy is hyperbolic. 
\end{theorem}

\begin{theorem}\label{thm:nec} A closed convex smooth plane curve with hyperbolic bicycling monodromy  has length   $L>2\pi$. 
\end{theorem}

\begin{rmrk}
The total curvature of a closed simple curve equals $2\pi$, hence its  average curvature equals $2\pi/L$. Theorem \ref{thm:nec} can be restated as a partial converse to Theorem \ref{thm:suf}: {\it A closed convex smooth plane curve with hyperbolic bicycling monodromy  has average curvature less than $1$.}
\end{rmrk}

In  Section \ref{sec:examples} below we show via  examples that the sufficient condition $|\kappa|\leq 1$ of Theorem \ref{thm:suf} is not necessary  for hyperbolicity of the monodromy, and that the necessary condition $L>2\pi$ of Theorem \ref{thm:nec} is  not sufficient.

\paragraph{Sketch of  the proofs of Theorems \ref{thm:suf} and \ref{thm:nec}.} 

Our main tool in  both theorems is  the ``hyperbolic development" interpretation of the no-skid condition  \eqref{eq:bei}, enabling the  translation of  these theorems to statements about  curves in the hyperbolic plane with periodic curvature function. Here is the main idea.

Given a smoothly parametrized unit speed    front track $\G:\R\to\R^2$, one considers
the flow  $b(t)\in \PSLt$ of the associated bicycling equation, i.e., $b(t)$ is the  Möbius transformation of $S^1$, mapping  $e^{i\theta(0)}\mapsto e^{i\theta(t)}$, where $\theta(t)$ is a solution  of equation \eqref{eq:bei}.

The flow $b(t)$ of the bicycle equation \eqref{eq:bei} appears as a curve in the group $\PSLt$ that starts at the unit element. One has a 2-fold covering $\SLt \to \PSLt$, and such a curve has a unique lifting as a curve in $\SLt$ also starting at the unit element. By Theorem \ref{thm:foote}, the lifted curve is  the flow of equation \eqref{eq:foote}, allowing us to treat the bicycle monodromy as an element of $\SLt$, rather than $\PSLt$, and we do so in many calculations below. 

Next, a {\em hyperbolic development} of $\G$ is a unit speed   parametrized curve in the hyperbolic plane, $\tG:\R\to\H$,with the same geodesic curvature as that of $\G$ at corresponding points,  $\kappa_{\tG}(t)=\kappa_{\G}(t)$ for all $t\in\R$. Note that $\tG$ is unique up to an orientation preserving isometry of $\H$. Let $h(t)\in\PSLt$  be the unique orientation preserving  isometry of $\H$ mapping $(\tG(0),\tG'(0))\mapsto (\tG(t),\tG'(t)).$

\begin{theorem}\label{thm:mon}
For any immersed  front track $\Gamma:\R\to \R^2$, $b(t)$ and $h(t)^{-1}$ coincide, up to a conjugation by a fixed element in  $\PSLt$.  
\end{theorem}

 This was proved  in \cite[\S2.10]{BLPT}. See Section \ref{sec:prelim}  below for another  proof, as well as a natural extension of this theorem to {\em piecewise smooth} front tracks $\G$ (we will use it for polygons, during  the proof of Theorem \ref{thm:nec}). 

 It follows that if $\G$ is a closed front track of length $L$ with bicycle  monodromy $b(\Gamma)\in\PSLt$, then $\tG$ is not closed in general, but since it has $L$-periodic curvature (same as that of $\G$), it satisfies $\tG(t+L)=h\cdot \tG(t)$ for all $t\in \R$ and some ``period map" $h\in\PSLt$, conjugate to $b(\G)^{-1}.$ In particular, $h$ is hyperbolic if and only if $b(\G)$ is.

Next, we need to translate properties of $\Gamma$ to  its hyperbolic development $\tG$. An  essential ingredient in the proof of Theorem \ref{thm:suf} is  the fact that the condition $|\kappa|\leq 1$ on a curve immersed in $\H$  implies that it is {\em embedded}. This is apparently due to 
W.~Thurston, but is not in the literature, so we supply a proof (whose essence was kindly communicated to us by  Martin Bridgeman). Once we know that $\tG$ is an embedded $h$-invariant curve,  it is fairly  easy to exclude a parabolic or elliptic $h$, so that $h$, and hence $b(\G)$, must be hyperbolic.

\mn 

The proof of Theorem \ref{thm:nec} is considerably more involved and occupies the bulk of the article. We start similarly  by looking at the  hyperbolic development $\tG$ of a given closed front track $\G$, convex with hyperbolic monodromy. Again, it follows from Theorem \ref{thm:mon} 
that $\tG$ has periodic curvature, with hyperbolic period map $h$. We then show, as in the case of Theorem \ref{thm:suf}, that $\tG$ is embedded, but this turns out to be more difficult.  

The key ingredient is a result in comparison geometry, called the ``Arm Lemma" (or perhaps the ``Bow Lemma", by some authors), stating a rather intuitive but not-so-easy to prove fact: given a convex arc in the euclidean plane, with a chord joining its end points (a line segment, the ``string'' of a ``bow"), the chord length increases in a hyperbolic development of the arc. 

This lemma excludes  immediately self intersection of points of $\tG$ which are less than $L$ apart along $\tG$. Then using invariance under a hyperbolic period map $h$, one can exclude arbitrary  self intersections of $\tG$. The problem with this argument is that the version of the Arm Lemma in the literature, while very general regarding ambient spaces, is available  only for polygons and it is not clear how  to use it to pass to the ``continuous limit". 
 
 Here we circumvent this difficulty via a rather lengthy ``bicycling" proof, again using Theorem \ref{thm:mon}. This theorem expresses hyperbolic development of Euclidean curves  in terms of the parallel transport of a certain connection along them ; in general, parallel transport along smooth curves is given as the  limit of  parallel transport along their polygonal approximation. See, for example,  App. 1A of \cite{A}, for this approach applied to the parallel transport of the Levi-Civita connection of a Riemannian surface. 

The next step is to use hyperbolicity of $h$ to prove  ``partial convexity'' of $\tG$, just enough to be able to apply the next step (we are certain full convexity holds, but tried  to avoid lengthening the proof any further). With this information on $\tG$, we apply a clever argument of M. Bridgeman from \cite{B}, using the  Gauss-Bonnet   and  isoperimetric inequalities in $\H$, to conclude that  $L>2\pi.$

\paragraph{Another approach to Theorem \ref{thm:nec}.} 

There is an alternative approach to the proof of Theorem \ref{thm:nec}, very different from the one in Section \ref{sec:nec} and inspired by contact geometry. The configuration space of directed unit segments in the plane is a 3-dimensional manifold, and the no skid constraint defines a contact structure on it. Bicycle motion defines a horizontal (Legendrian) curve in this contact manifold. The front and back tracks are two projections of this curve in the plane. 

If the bicycling monodromy is hyperbolic or parabolic then there exists a closed back track $\g$ (two of them, in the hyperbolic case). Generically, this back track is a cooriented wave front, a curve with a well defined oriented tangent line (given by the bicycle segment) at every point and, possibly, isolated semi-cubic singularities, that happen when the respective Legendrian curve is tangent to the fibers of the projection on the rear end of the bicycle segment. In particular, if $\g$ is oriented, one defines its integer-valued rotation number $\rho(\g)$.

The back track $\g$ uniquely defines the front track $\G$. We prove that $L \ge 2\pi |\rho(\g)|$ (Proposition \ref{prop:ineq}); if $\rho(\g) \neq 0$, this implies the inequality of Theorem \ref{thm:nec}. Based on numerous computer experiments, we conjecture that if $\G$ is smooth and strictly convex, and the bicycling monodromy is hyperbolic or parabolic, then $\rho(\g)=\pm 1$. We were unable to prove this conjecture so far. 

We also present an example of back tracks $\g$ with $\rho(\g) = 0$ such that  the respective front tracks $\G$ have arbitrary small lengths. In this example, the rotation number of $\G$ is zero. One could construct similar examples with arbitrary rotation numbers of the front tracks. 

We also consider variable length $\ell$ of the bicycle segment. Given a closed front track $\G$, the type of the bicycling monodromy depends on $\ell$. In particular, if $\ell \ll 1$, the monodromy is hyperbolic. If $\G$ is a closed strictly convex curve, we conjecture that there exists a number $\ell_0$ such that the monodromy is hyperbolic for $\ell < \ell_0$ and elliptic for $\ell > \ell_0$. This conjecture implies the above stated conjecture about the rotation number being $\pm 1$.

\paragraph{Contents of this article.} We start in Section \ref{sec:examples} with some simple examples  of bicycle monodromy of plane curves (rectangles and ellipses), showing that the sufficient  condition for hyperbolicity of Theorem \ref{thm:suf} is not necessary, and that the  necessary condition of Theorem \ref{thm:nec} is not sufficient.  Section \ref{sec:prelim} establishes the correspondence between bicycle transport and  hyperbolic development, our main tool. Theorem \ref{thm:suf} is proved in Section \ref{sec:suf}  and Theorem \ref{thm:nec} is proved in Section \ref{sec:nec}. In Section \ref{sec:alt} we describe another approach to prove Theorem \ref{thm:nec}, depending on a  conjecture on the rotation number of the closed back  tracks associated to front tracks with hyperbolic monodromy. We formulate  two related conjectures, interesting in their own right, which we hope to return to in future work.

\paragraph{Acknowledgments.} 
We acknowledge correspondence with Martin Bridgeman, who kindly supplied the proof of Lemma \ref{lemma:folk}, as well as useful discussions with  Maxim Arnold, Héctor Chang, Jesús Núñez,  Anton Petrunin and Agustí Reventós. GM and LH acknowledge support of CONAHCYT grant A1-S-45886. GB   acknowledges hospitality of the Toulouse Mathematics Institute during visits in 2023-4. LH is thankful for the hospitality of the Mathematics Department of the University of Santiago de Compostela, while parts of this article were done. ST was supported by NSF grant DMS-2404535 and by  Simons Grant TSM-00007747.

\section{Two examples}\label{sec:examples} 
We look at the bicycle monodromy of two  simple classes of curves: rectangles and ellipses. Rectangles can be studied precisely, ellipses numerically. 
Using these examples,  we verify  that, for  bicycle length  $\ell=1$:   (1) The sufficient conditions  $A>\pi$ (in the Menzin conjecture) and $\kappa\geq 1$ (in Theorem \ref{thm:suf}) are  not necessary  for the hyperbolicity of the bicycle monodromy; (2) the necessary condition $L>2\pi$ of Theorem \ref{thm:nec} is not  sufficient  for hyperbolicity of the bicycle monodromy.

\subsection{Rectangles}
Let $R_{a,b}$ be the rectangular path
$$(0,0)\mapsto (a,0)\mapsto(a,b)\mapsto(0,b)\mapsto (0,0),$$
where $a,b>0,$ and $b(R_{a,b})\in\SLt$ be its bicycle monodromy, based at the origin.

\begin{prop}
$$
\tr\left(b(R_{a,b})\right)=2-\sinh^2(a/2)\sinh^2(b/2).
$$
Thus, for all $a,b>0,$ $b(R_{a,b})$ is hyperbolic, parabolic or elliptic  if and only if\\ 
$\sinh(a/2)\sinh(b/2)>2,$  $=2$  or  $<2$ (respectively). 

\begin{figure}[h]
\centering
\def\svgwidth{.5\textwidth}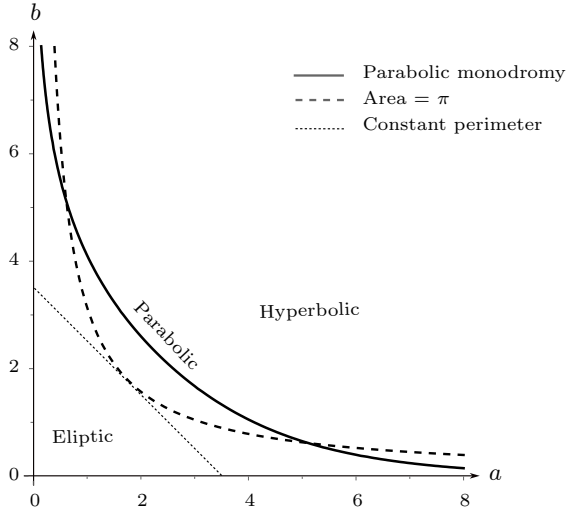
\caption{bicycling monodromy types of rectangles of size $a\times b$}
\end{figure}
\end{prop}
\begin{proof}
Let $H_a, V_b$ be the bicycle transport along horizontal and vertical segments, east and north, distances $a, b>0$  (respectively). In general, as noted in Theorem \ref{thm:foote} (see also Section \ref{sec:ricatti} below), the bicycle transport  $b(t)$ along a curve $\G(t)=(X(t),Y(t))$ from $0$ to $t$  satisfies 

$$b'=-{1\over 2}\mat{ X'&Y'\\ Y'&-X'}b,\quad b(0)=Id.$$ 
Thus 
\begin{align*}
    H_a&=\exp\left(-{a\over 2}\left(\begin{array}{cc}1&0\\ 0&-1\end{array}\right)\right)=\left(\begin{array}{cc}e^{-a/2}&0\\ 0&e^{a/2}\end{array}\right),\\
 V_b&=\exp\left(-{b\over 2}\left(\begin{array}{cc}0&1\\ 1&0\end{array}\right)\right)=
 \left(\begin{array}{rr}\cosh(b/2)&-\sinh(b/2)\\ -\sinh(b/2)
&\cosh(b/2)\end{array}\right),
\end{align*}
so
\begin{align*}b(R_{a,b})&=V_{-b}H_{-a}V_bH_a=
\frac{1}{2}\left(
\begin{array}{cc}
 1+ e^{-a}+\left(1- e^{-a}\right) \cosh b&  \left(1-e^a\right) \sinh b\\
 \left(1-e^{-a}\right)\sinh b &1+ e^{a}+\left(1- e^{a}\right) \cosh b 
\end{array}
\right).
\end{align*}
Hence
\begin{align*}
\tr(b(R_{a,b}))&=1+\cosh a+\cosh b- \cosh a\cosh b=2-(\cosh a -1 )(\cosh b-1)=\\
&=2-\sinh^2(a/2)\sinh^2(b/2).
\end{align*}

An element  $g\in \SL_2(\R)$ is elliptic, parabolic or hyperbolic if and only if $|\tr(g)|$ is $<,=$ or $ > 2 $ (respectively). From the last displayed formula,  $\tr(b(R_{a,b}))<2$ for all $a,b>0$, $\tr(b(R_{a,b}))=-2$ along  the curve $\sinh(a/2)\sinh(b/2)=2$,  $\tr(b(R_{a,b}))>-2$ below it and $\tr(b(R_{a,b}))<-2$ above it.
\end{proof}

We observe the following about rectangles $R_{a,b}$ with {\em parabolic} monodromy:

\begin{itemize}
\item As $a\to\infty$, $\mbox{area}(R_{a,b})\to 0$;
\item $\mbox{area}(R_{a,b})\leq \mbox{area}(R_{\cosh^{-1}(3),\cosh^{-1}(3)})=(\log(3+2\sqrt{2}))^2\approx 3.1$;
\item $\mbox{perimeter}(R_{a,b})\geq \mbox{perimeter}(R_{\cosh^{-1}(3),\cosh^{-1}(3)})=4(\log(3+2\sqrt{2}))\approx 7$.
\end{itemize}
(These follow directly from $\sinh(a/2)\sinh(b/2)=2$.)
Thus,
\begin{cor}\hfill
\begin{enumerate}[itemsep=-.2em, label={\rm (\arabic*)}]
\item There exist rectangles with hyperbolic monodromy whose area is arbitrarily small.  This shows that the sufficient condition $A(\G)>\pi$  in Menzin conjecture for $b(\G)$ to be hyperbolic is not necessary. Furthermore, no lower bound on the area enclosed by a front track $\G$ can be a necessary condition for hyperbolicity of $b(\G)$.

\item There exist rectangles with elliptic monodromy whose perimeter is arbitrarily big. Thus  the necessary  condition $L>2\pi$ in Theorem 
\ref{thm:nec} is not sufficient for hyperbolic monodromy, nor any lower bound on the perimeter is sufficient. 

\item The sufficient condition $|\kappa|\leq 1$ of Theorem \ref{thm:suf} for hyperbolic monodromy of non circular front tracks is not necessary, since there are rectangles (with unbounded curvature at the corners) with elliptic monodromy. 

\item Menzin's conjecture for rectangles: Any rectangle with area $>(\log(3+2\sqrt{2}))^2 \approx 3.1$, has hyperbolic monodromy.
\item A necessary condition for hyperbolic rectangles: A rectangle with hyperbolic monodromy has perimeter $>4(\log(3+2\sqrt{2}))\approx 7$.
\end{enumerate}
\end{cor}

\subsection{Ellipses}

One obtains similar results (albeit, numerically) for the bicycling monodromy of ellipses with semi-axes $a,b.$ 
In Fig.~\ref{fig:ellipses} we observe  the same trichotomy as before. 
\begin{figure}[h]
\centering
\def\svgwidth{1\textwidth}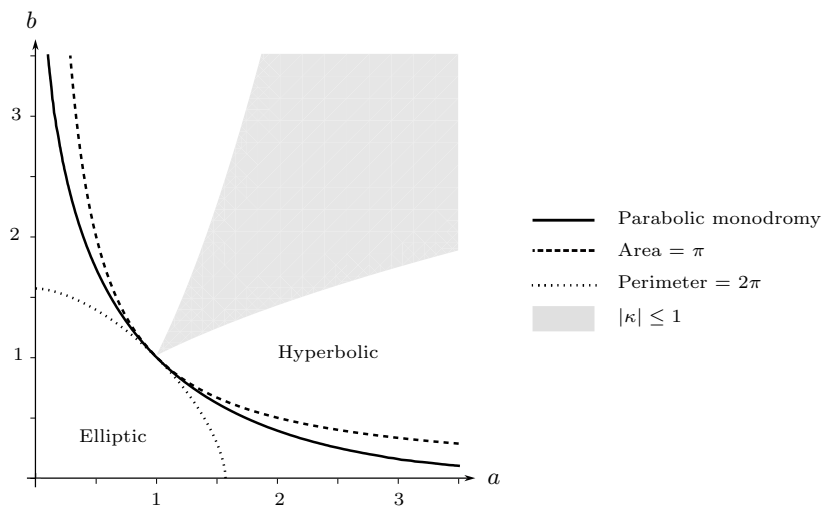
\caption{Monodromy types for ellipses with semi-axes $a,b$.}\label{fig:ellipses}
\end{figure}

\begin{enumerate}[left=0em, itemsep=-.2em, label={\rm (\arabic*)}]
\item There exist ellipses with hyperbolic  monodromy whose area is arbitrarily small. Thus the sufficient condition $A>\pi$ for hyperbolicity of the bicycle monodromy in Menzin conjecture is not a necessary condition.

\item There exist ellipses with  arbitrarily large perimeter whose monodromy is elliptic. Thus the necessary condition $L>2\pi$ of Theorem \ref{thm:nec} is not sufficient   for  hyperbolicity of the monodromy.

\item There exist ellipses with hyperbolic  monodromy with arbitrarily large  curvature (if $a\geq b$ then the maximum curvature, at $(a,0)$, is $a/b^2$). Thus the sufficient condition $|\kappa|\leq 1$ for hyperbolicity of the bicycle monodromy in Theorem  \ref{thm:suf} is not a necessary condition.  

\item
Menzin's conjecture for ellipses: any ellipse with parabolic monodromy has area $\leq \pi$, with equality at the unit circle; thus any ellipse with area $>\pi$  has hyperbolic monodromy.
\item Every parabolic ellipse has perimeter $\geq 2\pi$ (with equality at the unit circle); thus any ellipse with hyperbolic monodromy has perimeter $>2\pi$.
\end{enumerate}

We note that, unlike the case of rectangles, these items 4 and 5 are  particular cases of the general theorems concerning convex curves.

\section{Preliminaries}\label{sec:prelim}
Here we give a proof of Theorem \ref{thm:foote}, as well as a more detailed formulation of   Theorem \ref{thm:mon}, our main tool in proving Theorems \ref{thm:suf} and \ref{thm:nec}, extended also to piecewise smooth front tracks $\G$. More information can be found in \cite[\S2.10]{BLPT} and \cite[\S4-5]{F}.

\subsection{A linear version of the bicycle equation \eqref{eq:bei}}\label{sec:ricatti} 

Consider a linear system of ODEs, 
\be\label{eq:lin}
\y'=A(t)\y, 
\ee
where 
$$\y(t)\in\R^2,\quad A(t)\in Mat_{2\times 2}(\R).
$$
The flow of this system is the 1-parameter family of linear isomorphisms $g(t)\in\GL_2(\R)$, mapping $\y(0)\mapsto \y(t)$, where $\y(t)$ is a solution of \eqref{eq:lin}. It satisfies 
\be\label{eq:ling}
g'=A(t)g, \quad g(0)=Id. 
\ee
Since   $g(t)$  is linear, it acts  on the projective line (the space of lines through the origin in $\R^2$), 
$$\RP^1:=\left(\R^2\setminus \{ 0\}\right)/\R^*\cong S^1, 
$$ 
mapping $[\y]\mapsto [g(t) \y].$ The resulting  ODE on $\RP^1$ is  called the {\em projectivization} of the linear system \eqref{eq:lin}. 

Explicitly, identify $\R^2=\C$, and define 
\be\label{eq:hopf}
\pi:\C^*\to S^1, \quad \y=y_1+iy_2\mapsto e^{i\theta}={\y^2\over |\y|^2}={y_1^2-y_2^2\over y_1^2+y_2^2}+i {2y_1y_2\over y_1^2+y_2^2}.
\ee
\begin{prop}
Let $\y(t)\in\R^2$ be a non-vanishing solution of \eqref{eq:lin}, 
and let $e^{i\theta(t)}=\pi(\y(t))$, as in equation \eqref{eq:hopf}. Then $\theta(t)$ satisfies 
\be\label{eq:bep}
\theta'=p(t)\sin\theta +q(t)\cos\theta+r(t), 
\ee
where
$$p=d-a, \ q=b+c, \ r=c-b
$$
and $a,b,c,d$ are the entries of $A$, 
$$A(t)=\mat{a(t)&b(t)\\ c(t)&d(t)}.
$$
Conversely, every solution $\theta(t)$ of \eqref{eq:bep} lifts to a  non-vanishing  solution $\y(t)$ of \eqref{eq:lin}, unique up to multiplication by a non-zero constant.
\end{prop}
The proof is via a simple verification (omitted).

\mn 

Theorem \ref{thm:foote} follows from this proposition:  the bicycle equation \eqref{eq:bei}, associated with a front track $\G:[t_0, t_1]\to\R^2$,  is the projectivization  of  the linear system 
\be\label{eq:be}
\y'=A(t)\y,  \ \mbox{where}\ A=-{1\over 2}\left(\begin{array}{rr}X'&Y'\\Y'&-X'\end{array}\right) \mbox{ and } 
\G'=(X', Y').
\ee


\subsection{Hyperbolic development }
\begin{definition}\label{def:hyp_dev}
Let $\G:[t_0, t_1]\to\R^2$ be  a unit speed smooth immersion ($C^2$ is enough) and $\H$ the hyperbolic plane, of constant curvature $-1$. 
A  {\em hyperbolic development} of $\G$ is a unit speed smooth  immersion 
$\tG:[t_0, t_1]\to\H$ with  the same geodesic curvature as that of $\G$:
$$|\G'(t)|=|\tG'(t)|_{\H}=1, \ \kappa_{\tG}(t)=\kappa_\G(t)\ \mbox{ for all } t\in[t_0, t_1].
$$  
\end{definition}

Note that the {\em sign} of the geodesic curvature of an oriented curve on a surface depends on an orientation of the surface. Thus for this definition to make sense both $\R^2$ and $\H$ are assumed oriented.

Clearly, $\tG$ is determined by its initial conditions $(\tG(t_0), \tG'(t_0))$, 
and any two developments  differ by the orientation preserving  isometry of $\H$ that maps the initial conditions of one  to the other.

\begin{theorem}\label{thm:dev}
Let $\G:[t_0, t_1]\to\R^2$ be a unit speed smoothly  parametrized curve in $\R^2$ and $b:[t_0, t_1]\to\SLt$ the flow of the associated equation \eqref{eq:be}. 
\begin{enumerate}[left=0em,itemsep=0em, label={\rm (\alph*)}]
\item
Let $\tG:[t_0, t_1]\to\H$ be a hyperbolic development of $\G$ and   $h:[t_0, t_1]\to\SLt$
 the associated (lifted) curve of orientation preserving hyperbolic isometries, such that  $h(t)$ maps  $(\tG(t_0), \tG'(t_0))\mapsto (\tG(t), \tG'(t))$, $t\in[t_0,t_1]$,   and $h(t_0)=Id.$ Then  $h(t)$ coincides with $b(t)^{-1},$ up to an overall conjugation by an element of $\SLt$. 

\item More precisely, in the upper half-plane model of $\H$, $\tG(t):=b(t)^{-1}\cdot i$ is such that $\tG(t_0)=i,  \tG'(t_0)=i\overline{\G'(t_0)},$ and $\kappa_{\tG}=-\kappa_{\G}$. That is, $\tG$ is a development of $\G$ with the orientation of $\H$ opposite  to the one induced from its inclusion in $\R^2$.
\end{enumerate} 
\end{theorem}

\begin{proof} It is enough to prove (b). If $\G(t)=(X(t), Y(t))$, then its   geodesic curvature  is 
$$\kappa_\G=X'Y''-Y'X''.$$ 
Let 
$$h(t):=b(t)^{-1}=\left(\begin{array}{rr}a(t)&b(t)\\ c(t)&d(t)\end{array}\right)\in\SLt,
$$
where $b(t)$ is the  solution of \eqref{eq:be}. It satisfies 
\be\label{eq:h}h'=-h A(t), \quad h(t_0)=Id.
\ee
Proving (b) then amounts to  verifying that 
\be\label{eq:hs}\tG(t):=h(t)\cdot i= {a(t)i+b(t)\over c(t)i+d(t)}.
\ee 
satisfies
\be\label{eq:claims}
|\tG'|_\H=1,\quad 
\tG'(t_0)=i\overline{\G'(t_0)},  \quad 
\kappa_{\tG}=-\kappa_\G.
\ee
These  can be verified  by a straightforward  calculation  (we used Mathematica). Here are some details. 

Let $\tG(t)=(x(t), y(t))\in\H$ and recall that $\G(t)=(X(t), Y(t)).$ Using $ad-bc=1$ and equations  \eqref{eq:h}-\eqref{eq:hs}, one finds successively
$$y={1\over c^2+d^2}, \ x'+iy'=\frac{i\overline{\G'}}{(ci+d)^2}, \
x''+i y''= { i(c+i d) \overline{\G'}^2\over(c-i d)^3} -  {\overline{i\G''}^2\over (c-i d)^2}. 
$$
The first two formulas of \eqref{eq:claims} follow. A  formula for the geodesic curvature in $\H$ with respect to the standard orientation \cite[Prop 3, p 252]{D},    then gives
\be\label{eq:gc}
\kappa_{\tG}={x'y''-y'x''\over y^2}+{x'\over y}=X'' Y'-X' Y''=-\kappa_\G.
\ee

We now give a second proof of \eqref{eq:claims}, it will be useful later, during the proof of Theorem \ref{thm:dev_broken}. 

\begin{lemma}The  $\SLt$-action on the upper half-plane by M\"obius transformations assigns to every 
$$A=\mat{a&b\\ c& -a}\in \slt$$ the  Killing vector field 
$$v_A:=f\partial_z+\bar f\partial_{\bar z},\ \mbox{  where } f=b + 2 a z - c z^2.
$$ For the $A$ appearing in equation \eqref{eq:be} one has 
$$f={1\over 2}(1-z^2)Y'-zX'.$$
\end{lemma}
This is proved by a simple verification (omitted).

\mn 

Next, to find the speed and geodesic curvature  of $\tG(t)=h(t)\cdot i$ at some  $t=t_*,$ we translate $\tG$ by $h(t_*)^{-1}$  and  study  instead the curve $\g(\epsilon):=h(t_*)^{-1}h(t_*+\epsilon)\cdot i$ at $\epsilon=0$. Using equation \eqref{eq:h} and the above lemma,  one finds
\be\label{eq:inf}
\g'(0)=-v_A(i)= Y'(t_*)\partial_x+X'(t_*)\partial_y.
\ee
This has  norm 1,   which proves the first formula of \eqref{eq:claims}. Taking $t_*=t_0$  proves also the second one.  
To find $\kappa_\g(0)$, we calculate mod $\epsilon^3$, using equation \eqref{eq:h}, 
$$h(t_*)^{-1}h(t_*+\epsilon)=h^{-1}\left(h+\epsilon h'+{\epsilon^2\over 2}h''\right)=
I-\epsilon A+{\epsilon^2\over 2}\left(A^2-A'\right),$$
where we omit  evaluation at $t_*$ throughout. Now $A^2=(X'^2+Y'^2)I/4=I/4,$ so
\begin{align*}\g(\epsilon)&=h(t_*)^{-1}h(t_*+\epsilon)\cdot i=\left[I-\epsilon A+{\epsilon^2\over 2}\left({I\over 4}-A'\right)\right]\cdot i\\
&=i+\epsilon  \left(Y'+ iX'\right)+\frac{1}{2} \epsilon ^2 \left[2 X' Y'+ Y''+i(X'^2-Y'^2+ X'')\right].
\end{align*}
Thus
$$\g'(0)=(x',y')=(Y',X'), \  \g''(0)=(x'',y'')=(2 X' Y'+ Y'', X'^2-Y'^2+ X'').$$
Using the geodesic curvature formula  \eqref{eq:gc}  with $y=1$, one finds from the above
$$ \kappa_\g(0)=x'y''-y'x''+x'=Y'X''-X'Y''=-\kappa_\G,
$$
as needed. This proves Theorem \ref{thm:dev}. 

A less direct proof, using ``hyperbolic rolling,''  is given in  \cite[\S2.10]{BLPT}. 
\end{proof}

\subsection{Piecewise smooth curves}

Let $\G:[t_0, t_2]\to \R^2$ be a continuous and piecewise smooth unit speed immersion, with a single ``corner" at some  $t_1\in (t_0, t_2)$, so that $\G=\G_1*\G_2$ (concatenation), where $\G_1=\G\left|_{[t_0,t_1]}\right.$ and $\G_2=\G\left|_{[t_1,t_2]}\right.$ are smooth, with 
\be\label{eq:pw}
\G_2(t_1)=\G_1(t_1) \ \mbox{and }    \G'_2(t_1)=\rho\,\G'_1(t_1),\ \mbox{ for some  } \rho\in\SO_2.
\ee

A hyperbolic development of such a $\G$ is a continuous and piecewise smooth unit speed $\tG:[t_0,t_2]\to \H$ with the same data:  $\tG=\tG_1*\tG_2$,  where $\tG_1=\tG\left|_{[t_0,t_1]}\right.$ and $\tG_2=\tG\left|_{[t_1,t_2]}\right.$ are smooth developments of $\G_1, \G_2$ (respectively),  with 
$$\tG_2(t_1)=\tG_1(t_1)\ \mbox{and }  \tG_2'(t_1)=\rho\tG_1'(t_1).
$$
 See Fig.~\ref{fig:dev}. 

\begin{figure}[h]
\centering
\def\svgwidth{.7\textwidth}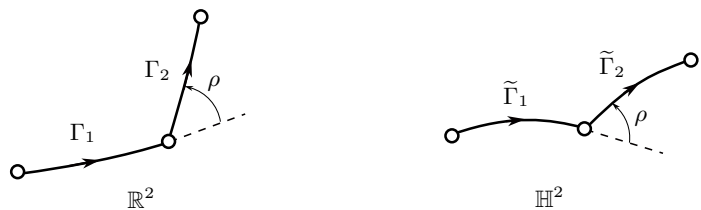
\caption{Hyperbolic development with a corner.}\label{fig:dev}
\end{figure}

Define the bicycle transport of $\G$ as the composition of the bicycle transports of $\G_1, \G_2$. Namely, let $$b_1:[t_0, t_1]\to\SLt,\ b_2:[t_1, t_2]\to\SLt$$ 
be the bicycle transports of $\G_1, \ \G_2$ (respectively), each satisfying an equation similar to \eqref{eq:be}, 
$$b_1'=A(t)b_1, \ b_1(t_0)=Id, \quad b_2'=A(t)b_2, \ b_2(t_1)=Id, $$
with $A(t)$ piecewise continuous, defined in $[t_0, t_1]$ by $\G_1$  and in $[t_1, t_2]$ by $\G_2$. 
Then let 
\be\label{eq:b}b(t):=\left\{\begin{array}{ll}
b_1(t)&\mbox{if }  t\in[t_0, t_1], \\
& \\
b_2(t)b_1(t_1) &\mbox{if }  t\in[t_1, t_2]. 
\end{array}\right.
\ee

These definitions extend in an obvious manner for an arbitrary piecewise smooth $\G$ with a finite number of   corners. 

\begin{theorem}\label{thm:dev_broken}
\leavevmode
\vspace{-.2em}
Theorem \ref{thm:dev} holds verbatim  in the piecewise smooth case. 
\end{theorem} 

\begin{proof} We treat the case of $\G:[t_0, t_2]\to \R^2$, with  a single  corner at $\G(t_1)$ for some $t_1\in (t_0, t_2),$  satisfying equation \eqref{eq:pw} for some $\rho\in \SO_2.$ The general case is similar and is omitted.  

We thus need to show that 
$$\tG(t):=b(t)^{-1}\cdot i, \ t\in [t_0, t_2],
$$
 is the hyperbolic development of $\G$ with initial conditions  $\tG(t_0)=i, \ \tG'(t_0)=i\overline{\G'(t_0)},$ and where $b(t)$ is defined via equation \eqref{eq:b}. One has  $\tG=\tG_1*\tG_2$, with $\tG_1=\tG\left|_{[t_0,t_1]}\right.$ and $\tG_2=\tG\left|_{[t_1,t_2]}\right.$. That is, 
 \begin{align*}
 \tG_1(t)&=b_1(t)^{-1}\cdot i,\ t\in [t_0, t_1], \\
 \tG_2(t)&=b_1(t_1)^{-1}b_2(t)^{-1}\cdot i,\ t\in[t_1,t_2].
\end{align*}
 Clearly, by Theorem \ref{thm:dev}, these are developments of $\G_1, \G_2$ (respectively). (Note that $\tG_2$ is the image under the hyperbolic isometry $b_1(t_1)^{-1}$ of the development  $b_2(t)^{-1}\cdot i$ of $\G_2$, hence is also a development of $\G_2$). It remains to show that $\tG'_2(t_1)=\rho\tG'_1(t_1).$ 
 Let us translate the common point $\tG_1(t_1)=\tG_2(t_1)$ to  $i$ by $b_1(t_1)$, defining 
 \begin{align*}
 &\g_1(\epsilon):=b_1(t_1)\tG_1(t_1+\epsilon)=b_1(t_1)b_1(t_1+\epsilon)^{-1}\cdot i,\ \mbox{where } \epsilon \leq 0,\\
 &\g_2(\epsilon):=b_1(t_1)\tG_2(t_1+\epsilon)=b_2(t_1+\epsilon)^{-1}\cdot  i,\ \mbox{where } \epsilon\geq 0. 
\end{align*}
 Let $A_i:=A(\G'_i(t_1)),$  $i=1,2.$ Then $\g_i'(0)=-v_{A_i}(i)=i\overline{\G'_i(t_1)}$  (see equation \eqref{eq:inf}). So $\g_2'(0)=\bar\rho\g'_1(0)$. It follows that $\tG'_2(t_1)=\bar\rho\tG'_1(t_1)$ (remember that $b(t)^{-1}\cdot i$ is a development of $\G$ with respect to the {\em  opposite} of the standard orientation). 
\end{proof}

\section{Proof of Theorem \ref{thm:suf}}\label{sec:suf}
Denote by $\H$ the hyperbolic plane and by $\kappa$ the geodesic curvature of an immersed curve. 
\begin{lemma}\label{lemma:folk}
An immersed  curve $\beta:\R\to\H$ with $|\kappa|\leq 1$ is embedded.
\end{lemma}

This lemma  appeared without proof in \cite[Lemma 1]{B}. The author has kindly sent us a sketch of the following proof, he heard in a class of W. Thurston.

\begin{wrapfigure}{r}{0.2\textwidth}
\vspace{-2em}
\def\svgwidth{.2\textwidth}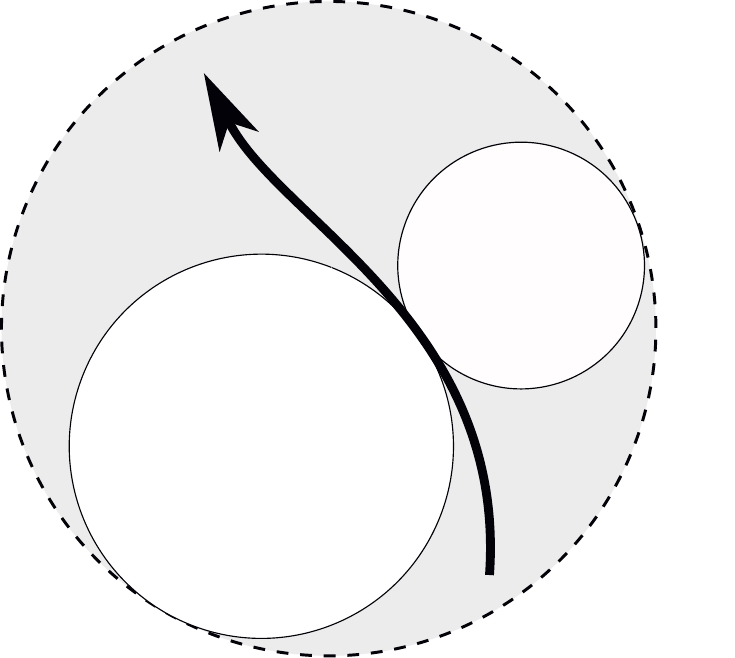
\end{wrapfigure}

\mn{\em Proof.} For each $t$ there are exactly two horocycles in $\H$ tangent to $\beta'(t)$ (curves with  $\kappa\equiv 1$). Since $|\kappa (t)|\leq 1$, the curve $\beta$ will never go inside either of the horocycles. Observe that the exterior of this pair of horocycles (the shaded region in the figure) consists of  two connected components: one that contains  ``the future'' of $\beta$ and the other one its past.

The geodesic $\g$ perpendicular to $\beta$ at $\beta(t)$ hits $\partial\H$ at the base points of these tangent horocycles. Again, $\g$ divides $\H$ in two disjoint parts, one that contains the future of $\beta$ and the other one its past. As $t$ increases, these perpendiculars $\g$ will always be disjoint, since their feet at ``infinity", $p_1$ and $p_2$,  are moving in opposite  directions along $\partial\H$ (or one stands  still and the other moves, when  $\kappa=1$, see below). 
In particular, $\beta$ cannot double back and self-intersect, i.e., it is an embedding.  

\begin{wrapfigure}{r}{0.3\textwidth}
\vspace{.2em}\hspace{1em}\def\svgwidth{.3\textwidth}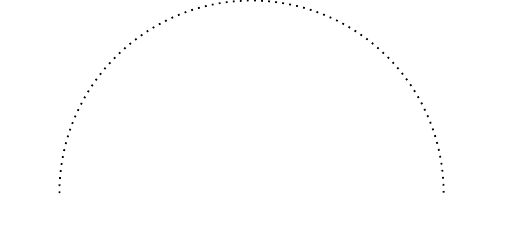
\end{wrapfigure}

 As for the claim about the motion of   $p_1$ and  $p_2$, we first note that this motion depends on the second-order jet of $\beta$, 
 and that at each point of $\beta$ one can find an osculating curve (i.e., with second-order contact with $\beta$)  with constant curvature. In the upper half plane model of $\H$, a curve with constant curvature $|\kappa|\leq 1$ is congruent to a line segment. If $|\kappa|<1$ then the segment is non-horizontal,  $p_1, p_2\in\R$ are the feet of a semicircle, and are easily seen to  move in opposite directions. 
 
\begin{wrapfigure}{r}{0.2\textwidth}
\vspace{-1.5em}
\hspace{1em}\def\svgwidth{.18\textwidth}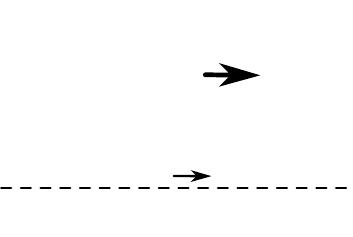
\end{wrapfigure}

 If  $|\kappa|=1$ then the line segment is horizontal (a horocycle), $\g$ is a vertical geodesic, with one foot  on the $x$-axis, and the other at $\infty$.  As the point on the horizontal segment moves, $\g$ is translated horizontally to a family of parallel (non-intersecting) vertical geodesics (these geodesics share a point at infinity).  
\qed

\vspace{1.5em} 

We now proceed to the proof of Theorem \ref{thm:suf}.  Assume $\G:\R\to\R^2$ is a unit speed  immersion, $L$-periodic, where $L$ is the length of $\G$, with bicycle monodromy $b(\G)$. Let $\tG :\R\to\H$ be a hyperbolic development  of $\G$, parametrized by arc length and with  the same geodesic curvature function as $\G$. Thus $|\kappa_{\tG}|\leq 1$, but $|\kappa_{\tG}|\not\equiv 1$. By Theorem \ref{thm:dev},
$\tG (t+L)=h (\tG (t))$ for all  $t\in\R$ and some $h\in\Isom^+(\H)\cong \PSLt$, conjugate to $b(\G)^{-1}$.

From  Lemma \ref{lemma:folk} we know that $\tG$ is an embedding. Let us see that $h$ cannot be an elliptic isometry.
First, $h$ cannot have finite order, else $h$-invariance of $\tG$  implies that $\tG$ is closed. Similarly, if $h$ is conjugate to an  irrational rotation then $\tG$ would  be dense in some annulus and self-intersect. See Fig.~\ref{fig:dense}. 

Thus $h$ is either parabolic or hyperbolic. In either case, for each $p\in\H$,  $h^{\pm\infty}(p):=\lim_{n\to\pm\infty}h^n(p)\in\partial\H=S^1$ exists, and $h$ is hyperbolic if and only if these two limits are distinct for some (and therefore all)  $p\in \H$. Take $p=\tG (0)$. The geodesic perpendicular to $\tG$ at $p$ divides $\partial\H$ in two disjoint open arcs. Since $\tG$ doesn't have curvature $\equiv 1$, it is not a horocycle, and therefore $h^\infty(p)=\lim_{t\to\infty}\G(t)$ does not coincide with either of the base points of the geodesic. Thus $h^\infty(p)$ belongs to the future arc.  Similarly, $h^{-\infty} (p)$ belongs to the past arc. Thus $h^\infty(p)\neq h^{-\infty}(p)$ and $h$ is hyperbolic.
\qed
\begin{figure}
\centering
\fig{.35}{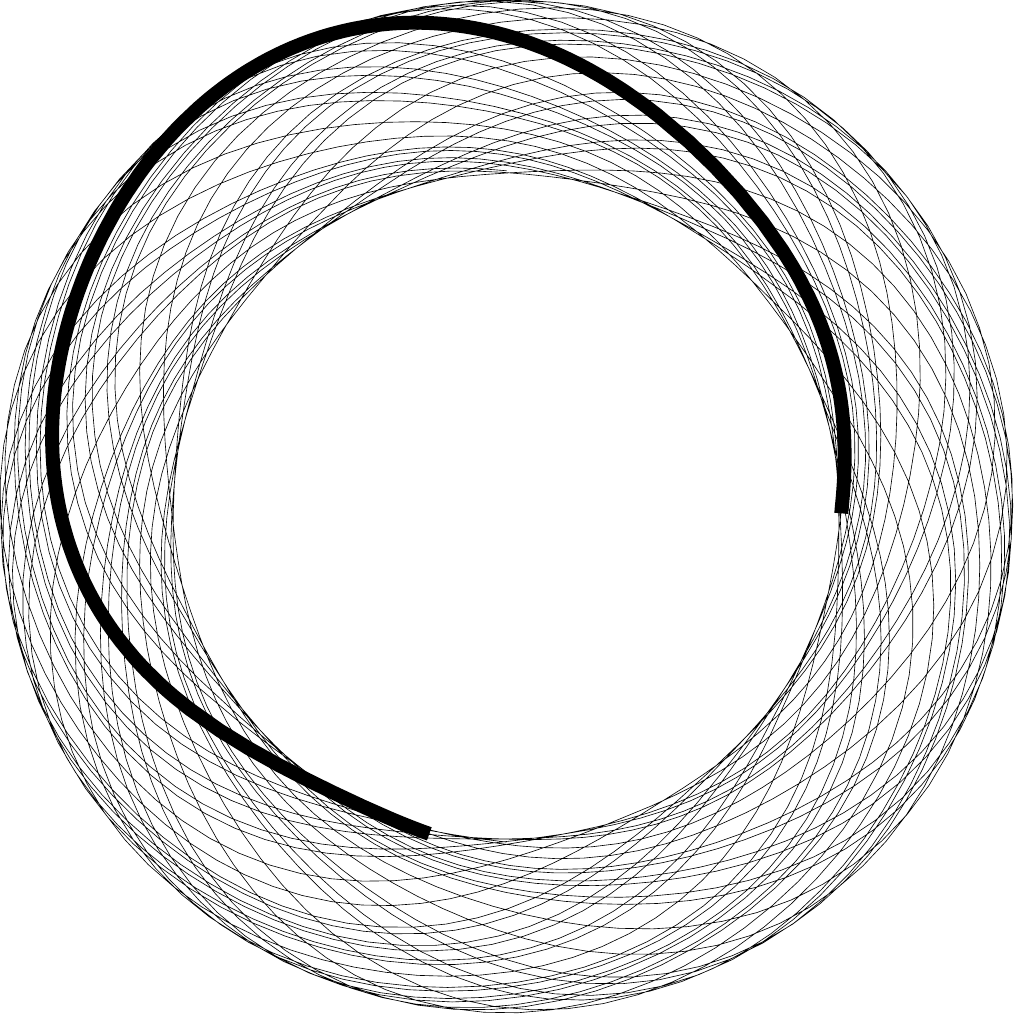}
\caption{A non-closed curve invariant under an elliptic isometry of infinite order (rotation by an irrational angle) fills up densely  an annulus and must self intersect.}\label{fig:dense}
\end{figure}

\begin{rmrk} Let us compare Theorem \ref{thm:suf} with the version of Menzin Conjecture proved in \cite{LT}. 
Both theorems provide sufficient conditions for a front track $\G$ to have hyperbolic monodromy, but under different assumptions. The assumptions of the former are local 
(small pointwise curvature), while the assumptions of the latter are global (convexity and large area). These two types of conditions are independent, although for simple curves, the local condition  $|\kappa|\leq 1$  of Theorem \ref{thm:suf} implies the area condition $A>\pi$. One way of showing this is by  proving that at least  one of the  osculating circles of a simple $\G$ is contained in its interior. Surprisingly, proving this is not so easy. See Theorem 2 of \cite{PZ} or Corollary 3.6 of \cite{W}.
\end{rmrk}

\begin{cor}\label{cor:small}
Let $\G\subset\R^2$ be a closed immersed curve. Then for sufficiently large $c>0$ (or sufficiently small bicycle length) the bicycle monodromy of $c\G$ is hyperbolic.
\end{cor}
\pf Since $\kappa_{c\G}={1\over c}\kappa_\G$, the result follows from the theorem.\qed

\section{Proof of Theorem \ref{thm:nec}}\label{sec:nec}
Let $\R\to\G$ be a unit speed parametrization of a smooth closed convex curve of length $L$ in the Euclidean plane with hyperbolic bicycling monodromy $b(\G)$. Let $\tG:\R\to\H$ be a hyperbolic development of $\G$ (see Definition \ref{def:hyp_dev}). By Theorem \ref{thm:mon}, 
$\tG$ is invariant under a hyperbolic isometry $h$ of $\H$, conjugate to $b(\G)^{-1}.$ That is 
$$\tG(t+L)=h(\tG(t)) \mbox{ for all } t\in\R.
$$
The action of $h$ on $\H$ leaves invariant a foliation by curves of constant curvature. In the upper half-plane model of $\H$, we can assume  that $h$ is  a dilation, $z\mapsto \lambda z,$ $\lambda>1$, and that the $h$-invariant foliation consists of the rays emanating from the origin.
We claim that $\tG$ is inscribed between two of these rays: this is clear for the compact portion $\tG([0,L])$, and by $h$-invariance,  for the whole $\tG$. 
The contact points of $\tG$ with  these two rays are clearly tangency points. By shifting if necessary the arclength parameter $t$, we can assume that $\tG$ is tangent to the right ray at $\tG(nL)$, $n\in\Z$. See Fig.~\ref{fig:tG}. 

\begin{figure}[h]
\centering
\def\svgwidth{.8\textwidth}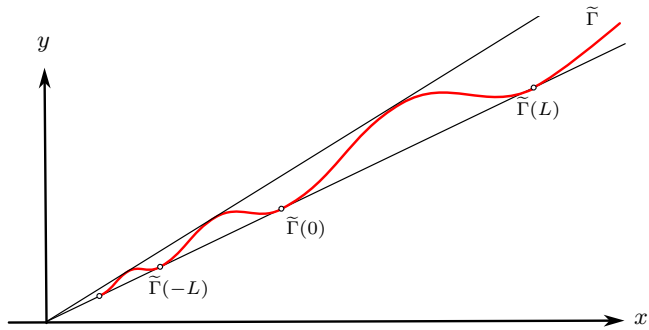
\caption{The hyperbolic development of a convex curve in $\R^2$ with a hyperbolic bicycle monodromy.} \label{fig:tG}
\end{figure}

The proof of the following proposition is postponed to the next subsection. 
\begin{prop}\label{prop:emb} Fix $n\in\N$, let $\g=\tG\left|_{[-nL, nL]}\right.$,  and let $g$ be the geodesic segment connecting the end points of $\g$. Then $\g$ is simple (no self intersections),  lies on one side of $g$, 
and intersects $g$ only at its end points,  $\tG(-nL)$ and $ \tG(nL).$\end{prop}

We now show how to use this proposition to prove Theorem \ref{thm:nec}. The proof,  following Bridgeman\cite{B}, uses the  Gauss-Bonnet Theorem and the isoperimetric inequality in $\H$.

\begin{figure}[h]
\centering
\def\svgwidth{.5\textwidth}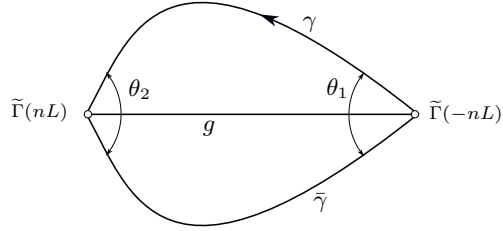
\caption{Bridgeman's picture.}\label{fig:bri}
\end{figure}

By Proposition  \ref{prop:emb} and the Jordan Theorem,  $\g\cup g$ is the boundary of a compact connected  domain in $\H$, homeomorphic to a disk, lying on one side of $g$.  Reflect $\g$ about $g$ to get a curve $\bar\g$, which together with $\g$ determines a simple closed curve, enclosing a (topological) disk $D$. Denote the internal angles at $\tG(-nL), \tG(nL )$ by $\theta_1, \theta_2$, respectively. See Fig.~\ref{fig:bri}. 

Applying the   Gauss-Bonnet Theorem to $D$, one gets
\be\label{eq:GB}
2\int_{-nL}^{nL} \kappa= A(D) +\theta_1 +\theta_2,
\ee
where $A(D)$ is the area of $D$. 

Let $A_p$ be the area of a disk in $\H$ with perimeter  $p$. Then 
$$
A_p=p\sqrt{1+\left({2\pi\over p}\right)^2}-2\pi.
$$
This follows from the usual formulas for the area and perimeter of a disk of radius $r$ in $\H$:
\begin{equation*}
\begin{split}
p&=2\pi\sinh r\\
A&=2\pi (\cosh r -1).
\end{split}
\end{equation*}

Now, the perimeter of $D$ is $4nL$, hence by the isoperimetric inequality, 
$$
A(D)\leq A_{4nL}=4nL\sqrt{1+\left({\pi\over 2nL}\right)^2}-2\pi<4nL+ {\pi^2\over 2nL}-2\pi,
$$
where we used the inequality $\sqrt{1+x}<1+x/2$ for $x>0$. 

The geodesic curvature of $\tG$ (same as that of  $\G$) is an $L$-periodic
 function, hence  combining the last  inequality with equation \eqref{eq:GB},  one gets
$$
2\pi=\int_0^L\kappa={1\over 2n}\int_{-nL}^{nL}\kappa={1\over 4n}\left[A(D)+\theta_1+\theta_2\right]
< L + {1\over 4n}\left({\pi^2\over 2 n L}+\theta_1 +\theta_2-2\pi\right), 
$$
where $\theta_1,\theta_2$ depend on $n$. 

\begin{lemma}\label{lemma:sum}
$\lim_{n\to\infty}(\theta_1+ \theta_2)<2\pi.$
\end{lemma}

\begin{wrapfigure}{r}{.5\textwidth}
\vspace{-1em}
\def\svgwidth{.5\textwidth}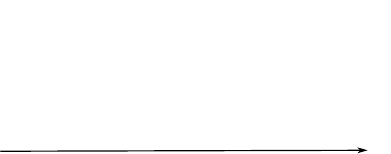
\end{wrapfigure}

\n{\em Proof.} Let us  go back to the upper half-plane picture of the beginning of the proof, as in Fig.~\ref{fig:tG}, so that 
$\tG(nL)=\lambda^n z_0,$ where $z_0=\tG(0).$ Then $g$ is the geodesic semicircle  passing through 
 $
\lambda^n z_0, z_0/\lambda^n.
$ 
The interior angles this circle forms with the ray $\R^+ z_0$ are $\alpha=\theta_1/2=\theta_2/2.$ 
As $n\to\infty$ the semi-circle $g$ tends to the positive $y$-axis and $\alpha$ tends to the angle $\alpha_\infty$ between $\R_+ z_0$ and the positive $y$-axis. Since $\alpha_\infty<\pi/2,$ one has $\lim(\theta_1+\theta_2)<2\pi. $
\qed

\mn

By Lemma \ref{lemma:sum}, for $n$ large enough,  $\theta_1 +\theta_2-2\pi<\delta_0$ for some $\delta_0<0$, hence for $n$ large enough, 
$${\pi^2\over 2 n L}+\theta_1 +\theta_2-2\pi <0,
$$
which implies $2\pi<L.$  This completes the proof of Theorem \ref{thm:nec}. \qed

\subsection{Proof of Proposition \ref{prop:emb}}\label{sec:emb}

The proof  is divided  into  three  steps: 
\begin{enumerate}[left=.5em,itemsep=0em, label={\bf \arabic*.}]
\item Show that 
$\tG\left|_{[0,L)}\right.$ is injective. We do not need for this step to assume that $b(\G)$ is hyperbolic. The main tool is a smooth version of a result in comparison geometry called the Arm Lemma. This is the most technical part of our proof of Theorem \ref{thm:nec}. 
\item Use  the hyperbolicity assumption on $b(\G)$  to show that $\tG$ is injective. 
\item Show that $\g$ lies on one side of $g$. 
\end{enumerate}

\mn{\bf Step 1.}  $\tG\left|_{[0,L)}\right.:[0,L)\to  \H$ is injective. 

 \mn The proof is based on  the following result, a special case of the so-called ``Arm Lemma"  \cite[\S9.63]{AKP}.  

\begin{lemma}
\label{lemma:arm}
Let $P\subset\R^2$  be an oriented convex closed polygon with vertices $q_0,  \ldots, q_n$,   edge lengths $d_0, \ldots,  d_n$ and external angles $\alpha_i\in(0,\pi)$ at $q_i$, $i=1,\ldots,n-1.$ Let $\tilde P\subset \H$ be a hyperbolic development of $P$, except for the last edge; that is,  
$$d_0=\tilde d_0,\ldots, d_{n-1}=\tilde d_{n-1},\quad  
 \alpha_2=\tilde  \alpha_2,\ldots,  \alpha_{n-1}=\tilde  \alpha_{n-1}.
 $$
See Fig.~\ref{fig:arm}. 
Let $\tilde d_n$ be the distance between $\tilde{q}_0$ and $\tilde{q}_n$. Then 
$d_n<\tilde d_n. 
$ \end{lemma}
 
\begin{figure}[h] \centering{\def\svgwidth{.8\textwidth}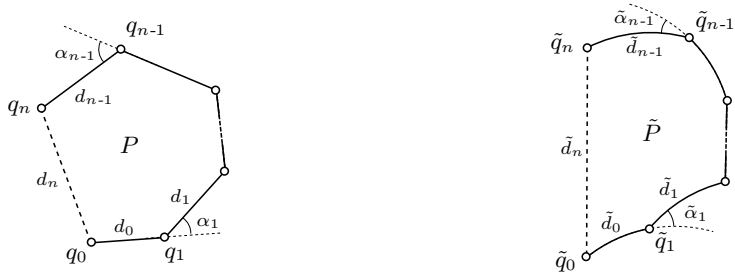}
\caption{The Arm Lemma.}\label{fig:arm}
\end{figure}

In the next subsection we shall use Lemma \ref{lemma:arm} to prove the following. 

\begin{lemma}[``Smooth Arm Lemma"]\label{lemma:sarm}
Let $\G:[a, b]\to\R^2$ be a unit-speed parametrization of a simple convex smooth  curve and $\tG:[a,b]\to\H$ a hyperbolic development of $\G$. That is,   $\G([a,b ])$  together with the chord connecting its endpoints forms a simple closed curve, the boundary of a convex set with non-empty interior, and 
$$|\G'(t)|=|\tG'(t)|_{\H}=1,\ \kappa_{\tG}(t)=\kappa_\G(t)\geq 0 \mbox{ for all } t\in[a,b].
$$
  Let $d, \tilde d$ be the distance between the end points of $\G, \tG$, respectively. Then $d\leq \tilde d.$\end{lemma}

Although Lemma \ref{lemma:sarm} seems a straightforward  ``continuous limit"  of   Lemma \ref{lemma:arm},
 we failed to come up with a simple proof  nor find one in the literature. Instead, we offer a rather indirect ``bicycle proof", using Theorem \ref{thm:dev_broken}.  See Subsection \ref{sect:arm} below.

\mn 

Step 1  follows from the smooth Arm Lemma:  applying  it  to $\G$ restricted to every  subinterval $[a,b]\subset [0,L)$, we  conclude that  $\tG(a)\neq \tG(b)$, hence $\tG\left|_{[0,L)}\right.$ is injective. \qed

\mn{\bf Step 2.} $\tG:\R\to \H$ is injective. 

\mn

We switch from the upper half-plane model of $\H$ in Fig.~\ref{fig:tG} to the Klein model, the unit disk  in $\R^2$, whose geodesics are chords of this disk, so geodesic convexity is the usual (euclidean) convexity.  Let $h\in\PSLt$ be the  isometry of $\H$ such that $\tG(t+L)=h(\tG(t))$, for all $t\in\R$. It is conjugate to the inverse of the bicycle monodromy $b(\G)$, so it is a hyperbolic isometry, with  two fixed points on $S^1=\partial\H$  (in the upper half-plane model we took it to be the  dilation $z\mapsto \lambda z$, where the 2 fixed points are $0, \infty$). We can assume, by passing if needed to a congruent $\tG$, that  the  two  fixed points of $h$ on $\partial\H$ are $(\pm 1,0)$, so the $h$-invariant foliation of $\H$  is the family of  ellipses tangent at $(\pm 1, 0)$, $x^2+y^2/b^2=1,$  $b\in (0,1).$ See Fig.~\ref{fig:klein}a. 

As before, $\tG$ is inscribed between two of these $h$-invariant ellipses (the shaded region in  Fig.~\ref{fig:klein}a), and the points $\tG(nL)$, $n\in\Z$, are tangency points of $\tG$ with the outer ellipse.  Denote by $\eta$ the arc of the outer ellipse between $\tG(0)$ and $\tG(L)$.
See Fig.~\ref{fig:klein}b.

\begin{figure}[h!]
\centering
\def\svgwidth{.85\textwidth}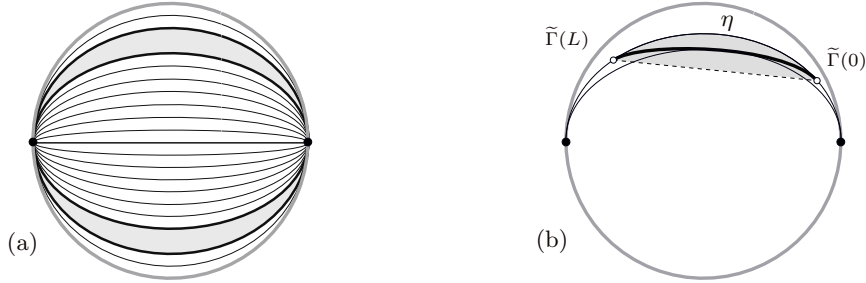
\caption{\s The  Klein model of $\H$:  (a) The $h$-invariant foliation of $\H$, with 2 $h$-fixed points on the boundary. $\tG$ is inscribed between two leaves of the $h$-invariant foliation, in the shaded area. (b) Lemma \ref{lemma:klein} states that $\tG([0,L])$ is contained in the shaded area.} \label{fig:klein}
\end{figure}

\begin{lemma}\label{lemma:klein}
$\tG([0,L])$ is contained in the convex hull of $\eta$ (the darkened region in Fig.~\ref{fig:klein}b). 
\end{lemma}

\begin{proof} 
The statement is clearly invariant under affine transformations, hence we can assume, by applying an affine map (a vertical stretching followed by a rotation), that the outer  ellipse is a circle, the boundary of the closed unit disk centered at the origin, see Fig.~\ref{fig:lemma}a. This disk  contains $\tG([0,L])$, $\eta$ is an arc of its boundary,  whose endpoints, $\tG(0)$ and $\tG(L)$,  are tangency points of $\tG$ with $\eta$, and have the same  $y$-coordinate. 
Thus, if $\tG(t)=(x(t), y(t))$, then $y(0)=y(L)$, $x(0)=-x(L)>0,$ $\tG(0)\cdot\tG'(0)=\tG(L)\cdot\tG'(L)=0,$ $y'(0)>0,$ $y'(L)<0$,  and we  want to show that $y(t)\geq y(0)$ for all $t\in(0,L).$  See Fig.~\ref{fig:lemma}a.

\begin{figure}[h!]
\centering
\def\svgwidth{.85\textwidth}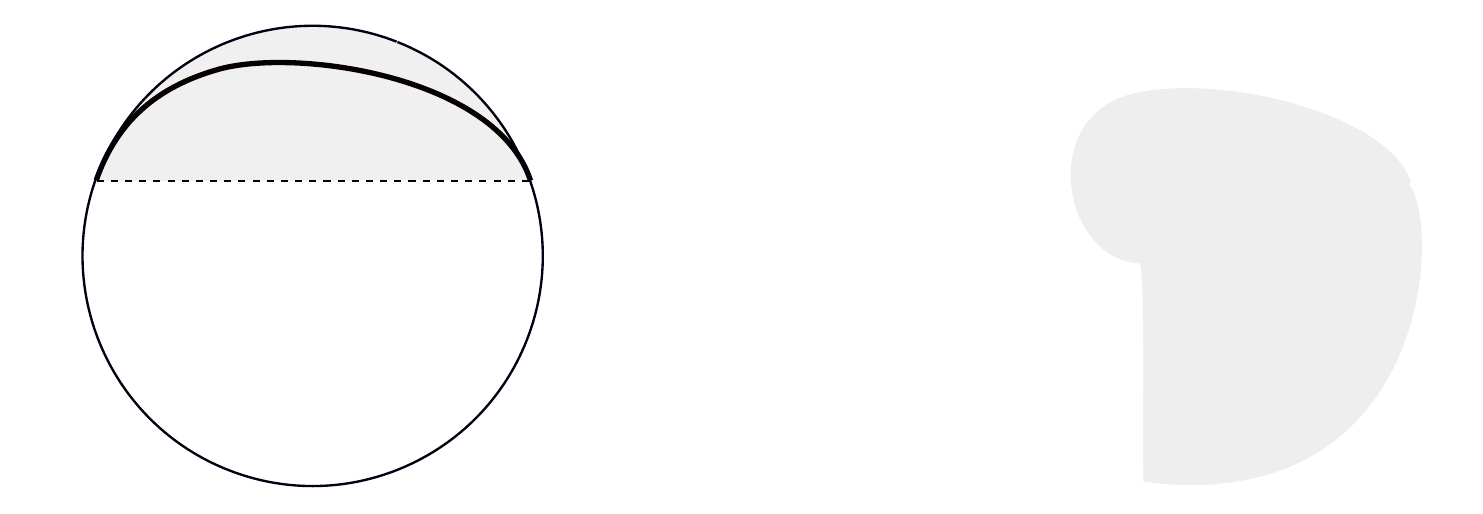
\caption{Lemma \ref{lemma:klein}.}\label{fig:lemma}
\end{figure}
 
 Assume, by contradiction, that $y_*:=\min\{ y(t) \; \big|\; t\in [0,L]\}<y(0)$ and let $t_*\in (0,L)$ such that $y_*=y(t_*).$ Clearly, $y'(t_*)=0,$ and since $\kappa\geq 0$, $x'(t_*)>0.$ Let $\Omega$ be the domain inside the unit disk bounded by the simple curve consisting of $\tG([0,t_*])$, followed by the vertical segment to the point of the unit circle below it, $S$, followed by the counterclockwise circular arc going along the boundary of the unit disk back to $\tG(0)$. See 
 Fig.~\ref{fig:lemma}b.  Now since $x'(t_*)>0$, then $\tG(t_*+\epsilon)$, for small enough $\epsilon>0$, lies inside $\Omega$, while $\tG(L)$ is outside it.
 It follows that $\tG(t)$ must intersect $\partial\Omega$ for some $t\in (t_*,L)$. Let $t_{**}\in (t_*, L)$ be the maximal such $t$, so $\tG(t_{**})\in \partial\Omega$ but $\tG(t)\not\in\partial\Omega$ for all $t\in(t_{**},L].$ Now $\tG(t_{**})$ clearly cannot belong to  the closed vertical segment from $\tG(t_*)$ to $S$, nor to the circular arc from $S$ to $\tG(0)$. So it must belong to $\tG([0,t_*])$, contradicting the injectivity of $\tG\left|_{[0,L)}\right.$ from  Step 1.
\end{proof}

\begin{wrapfigure}{r}{0.45\textwidth}
\vskip-1.5em
\def\svgwidth{.45\textwidth}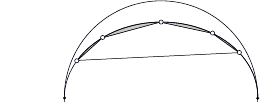
\end{wrapfigure}

\mn{\em Completing Step 2.} Lemma \ref{lemma:klein} shows that $\tG$ is injective since it follows, by $h$-invariance, that each curve segment $\tG([nL, (n+1)L])$, $n\in\Z$,  lies in the convex hull of the elliptic arc $h^n(\eta)$, and these regions are disjoint (except for meeting at their endpoints  $\tG(nL)$).  This completes Step 2. \qed

\mn{\bf Step 3.} $\g$ lies on one side of $g$.

 \mn 
 
 This step also follows immediately from Lemma \ref{lemma:klein}. The curve segment $\g:=\tG\left|_{[-nL, nL]}\right.$ lies between two curves: the   arc of the $h$-invariant ellipse between the end points of $\g$,  and the polygonal path with $2n+1$ vertices $\tG(k)$, $|k|\leq n$. Since both of these curves  lie on the same  side of the line segment  $g$ joining the endpoints of $\g$, the same holds for $\g$.  \qed.

\subsection{Proof of Lemma \ref{lemma:sarm} (the smooth Arm Lemma)}\label{sect:arm}
We start with the following  general setup. Let $\alpha$ be a $k\times k$ matrix-valued smooth  1-form in $\R^2$, i.e., $\alpha=fdx+gdy$, where $f,g:\R^2\to Mat_{k\times k}(\R)$ are smooth  functions. For  each continuous and piecewise-$C^1$  curve $\G:[0,1]\to \R^2$,   one considers the ODE 
\be\label{eq:pt}
\y'=A(t)\y,\mbox{ where }  A(t):=\alpha(\G'(t))\in Mat_{k\times k}(\R).
\ee

 Note that since $\G$ is assumed piecewise-$C^1$, $A(t)$ is only piecewise continuous.
 A solution to such an ODE is a continuous and piecewise-$C^1$ function $\y:[0,1]\to \R^k$, satisfying the ODE on each  subinterval  $[a,b]\subset[0,1]$ where $\g$ is $C^1$, and thus $A$ is continuous. The usual existence and uniqueness results for ODEs extend also to such a linear system with piecewise continuous coefficients; see, for example, exercise 1.2 of  \cite[p. 46]{Ha}.

Define $T(\G)\in \GL_k(\R)$ as the map that assigns to $\y(0)$ the value $\y(1)$ of the unique solution $\y(t)$ to equation \eqref{eq:pt} with this initial condition. Note that $T(\G)$ is invariant under orientation preserving reparametrizations of $\G$. 

We  can  think of $d\y-\alpha\y$ as the covariant derivative of a section of  the trivial vector bundle of rank $k$ on $\R^2$ with respect to a linear connection given by $\alpha$,  then  $T(\G)$ is the parallel transport of this connection along $\G$. 

\begin{prop}\label{prop:lim}Let $\G:[0,1]\to\R^2$ be a constant-speed smooth immersion
 and $\G_n\subset\R^2$, $n=1,2,3,\ldots,$  the polygonal path with $n+1$ vertices $q_0, \ldots , q_n,$ where $q_k=\G(t_k), $ $t_k=k/n,$ $k=0, 1, \ldots, n$. Then $T(\G_n)\to T(\G).$ 
\end{prop}

\mn{\em Proof.} This follows from   the next lemmas.

\begin{lemma}\label{lemma:d}
Let $\G:[a,b]\to \R^2$ be a unit-speed smooth immersion, with $d=|\G(b)-\G(a)|>0$ and $\G(b)=\G(a)+de^{i\theta_*}$ for some  $\theta_*\in\R$. Let   $\G'(t)=e^{i\theta(t)}$ for a smooth function $\theta:[a,b]\to\R$, whose variation is bounded by some $\epsilon<\pi/2$. That is, $|\theta(t_1)-\theta(t_2)|\leq \epsilon$ for all $t_1, t_2\in[a,b]$.   Then 
\begin{enumerate}[itemsep=0em, label={\rm (\arabic*)}]
\item  There is  $t_*\in[a,b]$ such that 
$\theta(t_*)=\theta_*$ {\rm (mod $2\pi$)}; 

\item    $L\cos\epsilon\leq d\leq L,$ where $L:=b-a$ is the length of $\G$. 
\end{enumerate}
\end{lemma}

\begin{wrapfigure}{r}{0.3\textwidth}
\vskip-1em
\def\svgwidth{.3\textwidth}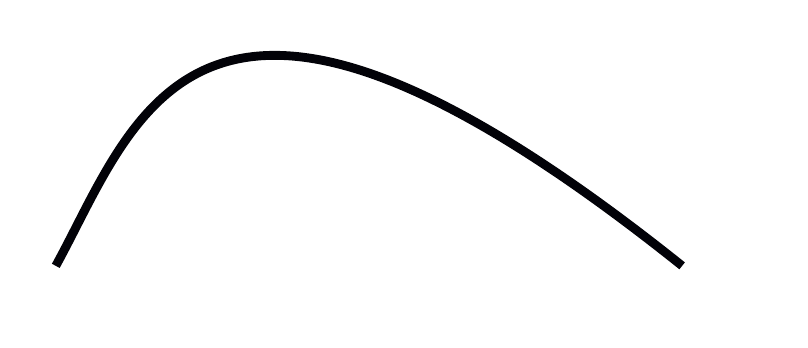
\end{wrapfigure}

\mn{\em Proof.} Let $\G(t)=(X(t), Y(t))$ and suppose without loss of generality that $Y(a)=Y(b)=0$ and $X(b)>X(a),$ so $d=X(b)-X(a)$ and $\theta_*=0$. Then 
$$d=X(b)-X(a)=\int_a^bX'(t)dt=\int_a^b\cos\theta(t)\,dt>0. 
$$

Let $t_*\in[a,b]$ be such that
$|Y(t_*)|=\max\{|Y(t)|\st t\in [a,b]\}.$ Then $\theta(t_*)=0$ or $\theta(t_*)=\pi$ (after shifting $\theta(t)$ by some integer multiple of $2\pi$). The second option is excluded because otherwise $\epsilon <\pi/2$ would imply  $\pi/2<\theta(t)<3\pi/2$   for all $t\in[a,b]$, thus   $\cos\theta(t)<0$, hence  $X(b)<X(a).$ Thus $\theta(t_*)=0$, which is  statement (1). This implies $|\theta(t)|<\epsilon$ for all $t\in[a,b]$, and since $\epsilon<\pi/2$ we have $\cos\theta(t)>\cos\epsilon$ for all $t\in[a,b]$. Using this in the above integral expression for $d$, one gets 
$d\geq L\cos\epsilon.$ The other inequality, $d\leq L$, is obvious. 
\qed

\begin{lemma}\label{lemma:approx}
Under the same assumptions of Proposition \ref{prop:lim}, assume also  that $\G_n:[0,1]\to\R^2$, $\G_n(t_k)=q_k$, and that  restricted to each $I_k:=[t_{k-1},t_k]$,  $\G_n$ is a constant speed parametrization of the line segment connecting $q_{k-1}$ to $q_k$, $k=1, \ldots, n$. Then $\G'_n\to\G'$ uniformly.
\end{lemma}

\begin{wrapfigure}{r}{0.22\textwidth}
\vskip-.5em
\def\svgwidth{.24\textwidth}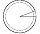
\end{wrapfigure}

\mn{\em Proof.} Let $L$ be the length of $\G$  so that  $|\G'|=L$. If $\epsilon>0$  then for  large enough $n$ the angle variation of  $\G'$ in each $I_k$  is $<\epsilon$. Let $q_k-q_{k-1}=d_k e^{i\theta_k}$. By statement (1) of Lemma \ref{lemma:d},  there is  $t_k^*\in I_k$ such that $\theta_k=\theta(t_k^*)$, hence $|\theta(t)-\theta_k|\leq \epsilon$ for all $t\in I_k$. By statement (2) of the same Lemma, $(L/n)\cos\epsilon \leq d_k\leq L/n.$ Now  one has $|\G_n'(t)|=d_kn$ for all $t\in I_k$, hence 
$L\cos\epsilon \leq |\G_n'(t)|\leq L.$ It follows that restricted to $I_k$ both $\G', \G'_n$ lie inside the anular region bounded  by  the concentric circles of radii $L$ and $L\cos\epsilon$ and  the angle $|\theta-\theta_k|\leq\epsilon$ (the shaded region in the figure). This region has diameter $\leq L\epsilon +L(1-\cos\epsilon)\leq 2L\epsilon $, thus $|\G'(t)-\G_n'(t)|\leq 2L\epsilon$ for all $t\in[0,1]$, as needed.  
\qed

\begin{lemma}\label{lemma:ct}Let $\y_0\in\R^k$, $A:[0,1]\to Mat_{k\times k}(\R)$ piecewise continuous and $\y(t)$ the solution of $\y'=A(t)\y,\ \y(0)=\y_0.$ Then $\y(1)$ depends continuously  on $A$. That is, if $A_n:[0,1]\to Mat_{k\times k}(\R)$ are piecewise continuous,  $\y_n(t)$ is the solution of $\y'_n=A_n(t)\y_n,$ $\y_n(0)=\y_0$, and $A_n\to A$ uniformly, then $\y_n(1)\to\y(1)$. 
\end{lemma}

\begin{proof}Integrating $
\y'=A\y,\  \y'_n=A_n\y_n
$
 over $[0,t]$, we obtain
$$\y(t)=\y_0+\int_0^tA\y,\quad  \y_n(t)=\y_0+\int_0^tA_n\y_n, \ \forall t\in[0,1], 
$$
thus
$$
\z_n(t):=\y_n(t)-\y(t)=\int_0^t\left( A_n-A\right) \y_n+\int_0^t A\z_n,
$$
so
\be\label{eq:zn}
|\z_n(t)|\leq \int_0^t|A_n-A| |\y_n|+\int_0^t |A| |\z_n|.
\ee

We next recall   {\em Gronwall's inequality}:
let $u,v:[0,1]\to\R$ be non-negative functions, with $v$ continuous and $u$ piecewise continuous, such that   for some $C\geq 0$ 
$$v(t)\leq C+\int_0^t uv \quad \mbox{ for all } t\in[0,1].
$$
Then 
$$\quad v(t)\leq C \exp \int_0^t u \quad \mbox{ for all } t\in[0,1].
$$
See, e.g., \cite[p.~24]{Ha}. (Note that in this reference $u$ and $v$  are both assumed to be continuous, but the proof applies unchanged to piecewise continuous $u$ as well.) 

\mn

Now from $\y_n(t)=\y_0+\int_0^tA_n\y_n$ one has $|\y_n(t)|\leq|\y_0|+\int_0^t|A_n||\y_n|,$ so by Gronwall's inequality, $|\y_n(t)|\leq |\y_0|\exp\int_0^t |A_n|.$ Since  $A_n$ is uniformly convergent it is uniformly  bounded, so the $\y_n$ are also uniformly bounded. 
Hence  the first summand in \eqref{eq:zn} tends to 0 as $n\to\infty$. It follows that if $\epsilon>0$ then  for large enough  $n$ the first  summand is $\leq \epsilon$, hence $|\z_n(t)|\leq \epsilon+\int_0^t |A| |\z_n|.$ Again by Gronwall's inequality,
$|\z_n(1)|\leq \epsilon\exp\int_0^1 |A|.$
Thus $\z_n(1)\to 0$, as needed. 
\end{proof}

The proof of the next lemma can be safely left to the reader. 

\begin{lemma}\label{lemma:exc}Let $T_n, T\in \GL_k(\R)$ such that $T_n(\y)\to T(\y)$ for every $\y\in\R^k$. Then $T_n\to T.$
\end{lemma} 

We can now complete the proof of  Proposition \ref{prop:lim}. By Lemma \ref{lemma:approx}, $\G'_n\to\G'$ (all convergence is uniform in $[0,1]$). Since $\alpha(\G')$ depends on $\G'$ continuously (even linearly),  $\alpha(\G_n')\to\alpha(\G').$ By Lemmas \ref{lemma:ct} and \ref{lemma:exc}, $T(\G_n)\to T(\G).$ \qed

\paragraph{Proof of the Smooth Arm Lemma \ref{lemma:sarm}.}We apply Proposition \ref{prop:lim} to
$$\alpha:=-{1\over 2}\mat{dX&dY \\ dY&-dX},
$$
so that  $T(\G)=b(\G)$. It follows that  given  a smooth convex $\G$,  there exists by Proposition \ref{prop:lim}  a sequence of convex polygonal approximations $\G_n$, with the same endpoints as $\G$, such that $b(\G_n)\to b(\G)$. Let $h_n:=b(\G_n)^{-1}$,  $h:=b(\G)^{-1}.$ 
By Theorem \ref{thm:dev}, part (b), there are hyperbolic developments $\tG, \tG_n$ of $\G, \G_n$ (respectively), starting at $i\in\H$ (using the upper half-plane model of $\H$), such that $h\cdot i$ and $h_n\cdot i$ are the other endpoint of $\tG,$ $ \tG_n$ (respectively). Thus the distance between the endpoints of these curves are $\tilde d_n=\dist(h_n\cdot i, i),$ $\tilde d=\dist(h\cdot i, i).$ By the (discrete) Arm Lemma, Lemma \ref{lemma:arm}, $\tilde d_n> d$ for all $n$. By Proposition \ref{prop:lim}, $h_n\to h$, hence $\tilde d=\lim \tilde d_n\geq d$. \qed

\section{Towards another proof of Theorem  \ref{thm:nec}}
\label{sec:alt}

In this section the bicycle length $\ell>0$ is arbitrary.

\subsection{Contact geometric viewpoint}

Recall some basic definitions of contact geometry. A contact element at a point $x \in M$ of a smooth manifold is a hyperplane in the tangent space $T_x M$. A co-orientation of a contact element is a choice of one (positive) side of this hyperplane. The space of contact elements is the projectivization of the cotangent bundle $PT^* M$, and the space of cooriented contact elements is its spherization $S T^* M$. The space of contact elements carries a contact structure, a completely non-integrable distribution of codimension 1. It is given by the constraint that the velocity of the base point of a contact element 
lies in the contact hyperplane.

We are concerned with the case when $M=\R^2$. In this case, a co-oriented contact element is a point of the plane and a co-oriented line through it. The co-orientation is given by a normal vector. The orientation of the plane and a co-orientation of a line define its orientation:
we assume that the co-orienting vector and the vector that orients the line, in that order, form a positive frame. 
The contact structure is defined by the same condition: the velocity of the point is aligned with the lines through it.

We can think of the bicycle as a cooriented contact element whose foot point is the rear end of the bicycle segment and whose line contains this segment. The orientation of the bicycle segment defines a co-orientation of the contact element as described above.

More concretely, if the coordinates of the rear end of the segment are $(x,y)$ and its direction is $\theta$, then a contact structure  is given  by the kernel of the differential 1-form 
\be\label{eq:contact}
\lambda=\cos\theta\ dy - \sin\theta\ dx.
\ee
 This defines the nonholonomic no skid constraint.

\begin{wrapfigure}{r}{0.25\textwidth}
\vspace{-1em}
\def\svgwidth{.25\textwidth}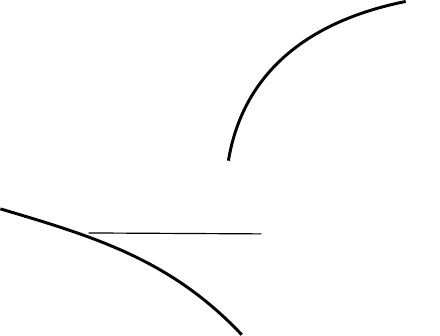
\end{wrapfigure}

Bicycle motion is then described as a smooth Legendrian curve in the space of directed segments of length $\ell$, that is, a curve in $ST^* \R^2$ that is everywhere tangent to the contact distribution. There are two projections to the plane, to the rear and to the front end of the segment. They are given by the formulas

$$
\pi_1: (x,y,\theta) \mapsto (x,y)$$
and
$$\pi_2: (x,y,\theta) \mapsto (X,Y)=(x - \ell \cos\theta, y - \ell \sin\theta).
$$
The fibers of $\pi_1$ are Legendrian, and the fibers of $\pi_2$ are transverse to the contact distribution. It follows that the second projection of a smooth Legendrian curve, the front bicycle track,  is smooth; whereas the first projection, the rear bicycle track, may have singularities. 

Singularities happen when the Legendrian curve is tangent to the fibers of $\pi_1$. Generically, this is an isolated  quadratic tangency, in which case the rear track $\g$ has a semi-cubical singularity (in bicycle terms, the rear wheel stops momentarily and the front wheel moves in a direction perpendicular to the bicycle segment).
Thus a generic rear track is a wave front, that is, a  curve that has a well defined tangent line at every point, and whose only singularities (if any) are semi-cubical cusps.

For a reference on Legendrian knots, see \cite{Et}, and on geometry and topology of wave fronts, see \cite{Ar}. 

\subsection{Relations between the rear and front tracks} 

Let $\g$ and $\G$ be oriented closed rear and front tracks, that is, the bicycle monodromy of $\G$ is hyperbolic or parabolic. 
Denote by $\rho(\g)$ and $\rho(\G)$ the rotation numbers of these curves. 

Let us define these numbers. The front track  is a closed smoothly immersed oriented curve $\G(t)$ in $\R^2$, and   $\rho(\G)$ is the winding (or total turning) number of its velocity vector  $\G'(t)$. The  back track $\g$ is given by the motion of a unit vector $\r(t),$ the ``bicycle frame", so that  $\g(t)=\G(t)+\ell\r(t).$
The rotation number  $\rho(\g)$ is, by definition, that  of  $\r(t)$. The bicycle frame is  co-oriented by $i\r(t)$ (the ``right-pedal"). The pair $(i\r(t),\g(t))$ defines a Legendrian curve $\tg(t)$ in the space of co-oriented contact elements $ST^* \R^2= S^1 \times \R^2$.

When $\g$ is smooth, $\rho(\g)$ is the winding number of $\g'(t)$. If $\g$ is a generic wave front with cusps, the number of full turns of the directed bicycle segment may be a half-integer, but if the front is co-oriented, this is still an integer. More generally, 
$\rho(\g)$ is defined also when  $\g$ is non-generic, for example, it  might even reduce to a single point. 

The topological definition of $\rho(\g)$ is as follows. As we explained earlier, one has an oriented closed Legendrian curve in the space of co-oriented contact elements $ST^* \R^2= S^1 \times \R^2$, that is, a map of a circle to this space. The projection on the $S^1$ factor provides a map of an oriented circle to an oriented circle, and $\rho(\g)$ is its degree. 

The back track $\g$, at its smooth points, is both oriented and cooriented. It is oriented by $\G$ and  co-oriented by the bike frame via the ``right-pedal-rule''. The generic singularities of $\g$ are semi-cubical cusps and the co-orientation is continuous across them. The orientation is reversed when crossing such a cusp. 

\begin{wrapfigure}{r}{0.3\textwidth}\vspace{-1em}\def\svgwidth{.28\textwidth}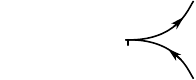\end{wrapfigure}

 The  contact distribution on the space of co-oriented contact elements of the plane is oriented by the contact form $\lambda$ of  equation \eqref{eq:contact} (the restriction of $d\lambda$ to the contact distribution is an area form on it). In fact, it is parallelized. The parallelization is given by the oriented fibers of the projection $\pi_1$. The tangent line to a closed smooth oriented Legendrian curve makes a number of turns with respect to this parallelization. This is the {\em Maslov index} $\mu$ of a Legendrian curve. 
It  equals the algebraic number of times the curve is tangent to the fibers of the projection $\pi_1$. That is, the 
Maslov index  $\mu(\g)$ is  the algebraic number of cusps, where a cusp is positive if it is traversed in the co-orienting direction and negative otherwise. See Fig.~\ref{fronts} for two  examples.

\begin{figure}[ht]
\centering
\def\svgwidth{.9\textwidth}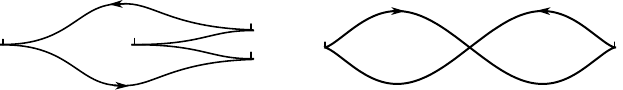
\caption{Two back tracks with $\rho =0$ and  $\mu =2$. }
\label{fronts}
\end{figure}

The relation between the rotation numbers of the rear and front tracks is given by the next lemma from \cite{LT}. 

\begin{lemma} \label{lm:rot}
\be\label{eq:form}
\rho(\G)=\rho(\g) + \frac{1}{2} \mu(\g).
\ee
\end{lemma}

For convenience of the reader we provide a proof. 

\begin{proof}
Assume that $\g$ is a generic wave front. We shall prove the result in this case; the degenerate case follows by continuity. 

The idea is to consider a 1-parameter family of front tracks with the same rear track and variable length of the bicycle $\ell$. The rotation number of the front track remains the same in this family, so to establish the result it suffices to consider 
a very short bicycle with $\ell \ll 1$.

Along a smooth arc of $\g$, the curve $\G$ is $C^1$-close to $\g$, and the revolution of the tangent lines is the same. At the cusps, $\G$ is a smoothing of $\g$, and its rotation number gains or looses half a turn.
There are four cases, depending on the orientation of the front and the sign of the respective Maslov index, illustrated in Fig.~\ref{fig:cusps}. The result follows.
\end{proof}

\begin{figure}[h]
\centering
\def\svgwidth{.8\textwidth}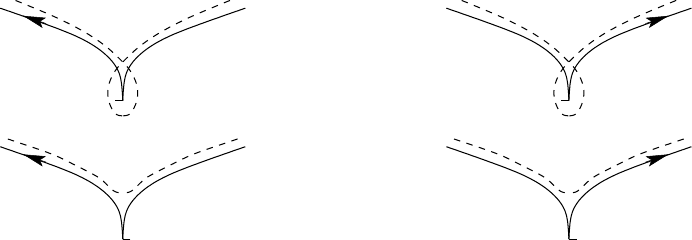
\caption{Proof of Lemma \ref{lm:rot}. The three numbers in each case denote the changes  in $\rho(\G), \rho(\g), \mu(\g)$ (respectively) due to traversing a cusp, verifying that these changes satisfy formula \eqref{eq:form} of Lemma \ref{lm:rot}.}\label{fig:cusps}
\end{figure}

Let $L$ be the length of the front track $\G$. One has

\begin{prop} \label{prop:ineq}
$L \ge 2\pi\ell |\rho(\g)|$, with equality if and only if $\G$ is a circle of radius $\ell$, perhaps traversed several times, and $\g$ is its center.
\end{prop}

\begin{proof}
Assume that $\G$ is parameterized by arc length, that is, $\G'=(X',Y')$ is a unit vector.
The bicycle differential equation reads
$$
\ell \theta' =  X' \sin\theta-Y' \cos\theta.
$$
The right hand side is the dot product of two unit vectors, hence 
$\ell |\theta'| \le 1$, with equality if and only if $\G'$ is orthogonal to the bicycle segment. 

Since 
$$
2\pi \rho(\g) = \int_0^L \theta'(t)\ dt,
$$ 
one has
$$
2\pi \ell |\rho(\g)|  \le  \int_0^L \ell|\theta'(t)|\ dt \le \int_0^L dt = L.
$$

The second inequality is equality when  $\ell |\theta'| = 1$ for all $t$, that is, the front track is always perpendicular to the bicycle segment. This means that the rear wheel does not move, $\G$ is a (multiple) circle of radius $\ell$, and $\g$ is its center. Then $L=2\pi\ell |\rho(\g)|$.

Conversely, if $\G$ is a circle of radius $\ell$, traversed $n$ times, and $\g$ is its center, then $L=2\pi n$.
\end{proof}

Thus if a closed front track $\G$ has a hyperbolic or parabolic monodromy and the rotation number of the respective rear track does not vanish, one has a lower bound on the length of $\G$, linear in the bicycle length $\ell$.

\begin{rmrk}
{\rm We note that if the monodromy is hyperbolic, then the two respective rear tracks have the same rotation numbers. Indeed,  consider the initial positions of the bicycle frame corresponding to the two fixed points of the monodromy. As these two bicycle segments of length $\ell$ move around $\G$, they never coincide, therefore they make the same number of full turns.
}
\end{rmrk}

However, if $\rho(\g)=0$, one can make the length of $\G$ arbitrarily small for every $\ell$.
For example, let $\g$ be the rear track depicted in Fig.~\ref{saucer}. 
The respective front track $\G$ is regularly homotopic to an eight-shaped curve. 
By contraction in the vertical direction, one can make the wave front $\g$ as flat, that is, as $C^1$-close to a doubly traversed segment, as one wishes. Then $\G$  is also close to a doubly traversed segment of the same length, translated a distance $\ell$. Further scaling down $\g$, one can make $L$ as small as one wishes. 

\begin{figure}[ht]
\centering
\def\svgwidth{.7\textwidth}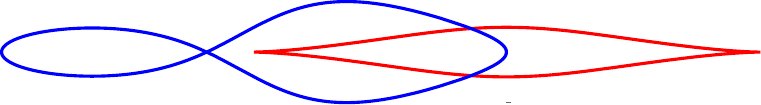
\caption{The simplest wave front, the ``flying saucer" $\g$, and the associated front track $\G$ (blue).  Both $\rho(\g)$ and $\mu(\g)$ vanish, hence also  $\rho(\G)$, by equation \eqref{eq:form}.}
\label{saucer}
\end{figure}

\subsection{Two conjectures} 

The following conjectures are based on computer experiments. 

\begin{conjecture} \label{conj:1}
If $\G$ is a simple strictly convex curve with hyperbolic or parabolic monodromy, then $\rho(\g)=1$.
\end{conjecture}

In view of Proposition \ref{prop:ineq}, this conjecture implies Theorem 2. It also implies that if the monodromy is parabolic, then $L\geq 2\pi$ with equality iff $\G$ is a circle of radius 1.

\mn 

Fix a closed front track $\G$ and consider the dependance of the type of the monodromy (hyperbolic, parabolic, elliptic) on the bicycle length $\ell$. If $\ell \ll 1$, the monodromy is hyperbolic: this is intuitively clear -- the bicycle segment remains nearly tangent to $\G$, see Corollary \ref{cor:small} for a proof. As $\ell$ increases, the type of the monodromy may change; if $\G$ bounds a non-zero area, the monodromy  eventually becomes elliptic, as the theory of  the hatchet planimeter shows, see \cite{F, FLT}. 

\begin{conjecture} \label{conj:2}
If $\G$ is a simple strictly convex curve, then there exists a number $\ell_0$ such that if $\ell < \ell_0$ then the monodromy is hyperbolic, and if $\ell > \ell_0$ it is elliptic.
\end{conjecture}

Conjecture \ref{conj:2} implies Conjecture \ref{conj:1}. Indeed, for $\ell \ll 1$, the rear track is also a simple strictly convex curve, and hence $\rho(\g)=1$. As $\ell$ increases, the rotation number changes continuously, hence $\rho(\g)=1$ persists until $\ell$ reaches the break point $\ell_0$. 

However,  the dependence of the type of the monodromy on the bicycle length $\ell$ may be more complicated for other front tracks.

For example, let $\g$ be a back track with $\rho(\g)=0$ and  $\mu(\g)=2$ (two such examples  are shown in Fig.~\ref{fronts}).  Let $\G$ be the corresponding front track with  $\ell=1$. By Lemma \ref{lm:rot},  $\rho(\G)=1.$  For sufficiently small value of $\ell$, the monodromy of $\G$ is hyperbolic and the rotation number of a respective rear track  $\g$ is that of $\G$, i.e., $\rho(\g)=1$ (see Corollary \ref{cor:small} and \cite[Prop. 3.15]{BLPT}). 
By construction, the monodromy of $\G$ with $\ell=1$ is also hyperbolic (or parabolic), but $\rho(\g)=0$. Since $\rho(\g)$ varies continuously with respect to $\ell$ in an interval with hyperbolic or parabolic monodromies, this implies that the monodromy of $\G$ must become elliptic for some  $\ell \in (0,1)$.

We illustrate this with the second example of Fig.~\ref{fronts}, where $g, \G$ are explicitly parametrized  by 
 \be\label{eq:exact}\g(t):= \left(2\cos t,\sin (2t)\sin ^2t \right),\quad \G(t)=\g(t)+\sqrt{{2\over 3-\cos (6 t)}}\left(-1, \sin (3 t)\right).
\ee
The monodromy of $\G$ is then calculated numerically, and is found to be hyperbolic for $\ell\in (0.9272, 1)\cup(1,\infty),$ elliptic  in the interval $\ell\in(0.9272, 1)$, with parabolic endpoints (the value $\ell=1$ is precisely parabolic, $\ell=0.9272$ is approximate). In the interval $\ell\in (0,0.9272]$ the rotation numbers of back tracks are $\rho(\g)=1$ and the Maslov index is $\mu(\g)=0$, and in the interval $\ell\in[1,\infty)$ one has $\rho(\g)=0$ and $\mu(\g)=2$, so that in all cases equation  \eqref{lm:rot} is satisfied. See Fig.~\ref{fig:fish}. 

\begin{figure}[H]
\centering
\def\svgwidth{.9\textwidth}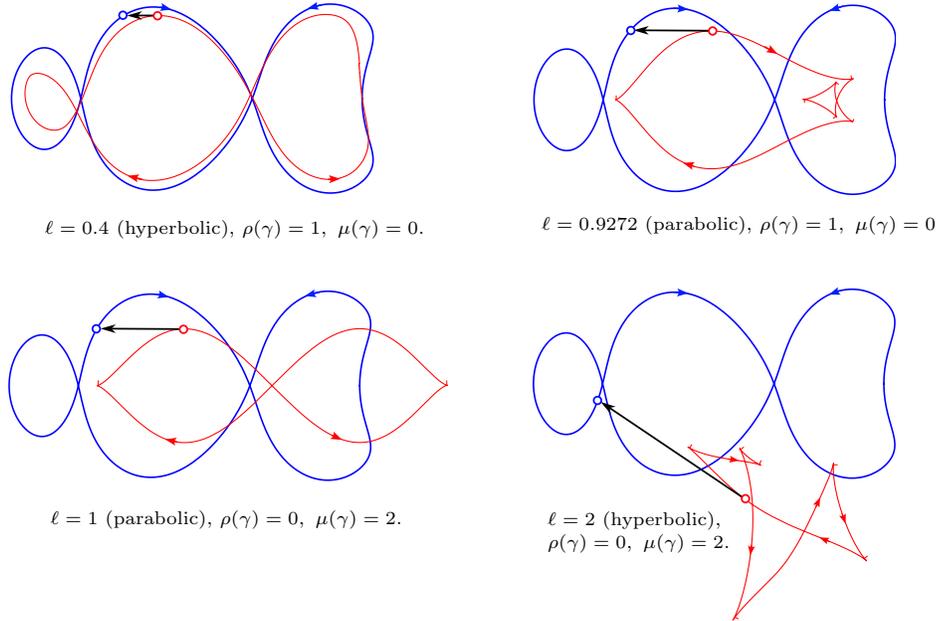
\caption{An example of  a non-convex front track $\G$ (blue), given by equation \eqref{eq:exact},  with $\rho(\G)=1,$ whose set of hyperbolic bike lengths $\ell$ consists of two disconnected intervals, with a ``gap" of elliptic bike lengths $(0.9272, 1).$  The back tracks $\g$ (red) with bike lengths to the left of the gap have $\rho(\g)=1$ and $\mu(\g)=0$ (the upper two cases), and to the right of the gap have $\rho(\g)=0$ and $\mu(\g)=2$ (the lower two cases). The back track for $\ell=1$ is given by formula \eqref{eq:exact}, the other three were found by solving the bicycle equation \eqref{eq:bei} numerically. Conjecture  \ref{conj:2} states that this gap does not occur  for {\em strictly convex} front tracks $\G$.}\label{fig:fish}
\end{figure}

\end{document}

%% file: preamble.tex
\usepackage{amssymb,amsfonts,amsmath,amsthm,bm, color,enumerate,enumitem,epsfig,float,graphbox,graphics,graphicx,hyperref,import,marginnote,
mathrsfs, mathtools,multicol,pdfpages,psfrag,tasks,transparent, url, wrapfig}
\usepackage[all]{xy}

\usepackage[font=small,labelfont=bf]{caption} 

\usepackage[margin=10pt,font=small,labelfont=bf,belowskip=1em, 
labelsep=period]{caption}

\newcommand{\fig}[2]{\includegraphics[width=#1\linewidth]{figures/#2.pdf}}

\newtheorem{lemma}{Lemma}[section]
\newtheorem{prop}[lemma]{Proposition}
\newtheorem{theorem}{Theorem}
\newtheorem{cor}[lemma]{Corollary}
\newtheorem{conjecture}[lemma]{Conjecture}

\newcommand{\Isom}{{\rm Isom}}

\theoremstyle{definition}
\newtheorem{definition}[lemma]{Definition}
\newtheorem{rmrk}[lemma]{Remark}

\newcommand{\nmo}{{n\mkern1mu\text{-}\mkern-1mu1}}

\newcommand{\mat}[1]{\left(\begin{array}{rr}#1\end{array}\right)}

\newcommand{\PSLt}{\mathrm{PSL}_2(\R)}
\renewcommand{\r}{\mathbf{r}}

\newcommand{\N}{\mathbb N}

\newcommand{\tG}{\tilde \Gamma}

\newcommand{\y}{{\bm y}}
\newcommand{\z}{{\bm z}}

\newcommand{\G}{\Gamma}

\newcommand{\tr}{\mathrm{tr}}

\newcommand{\g}{\gamma}
\newcommand{\tg}{{\tilde{\gamma}}}

\newcommand{\pf}{{\n\em Proof. }}

\newcommand{\RP}{\mathbb{RP}}

\newcommand{\st}{\, | \,}

\newcommand{\R}{\mathbb{R}}

\newcommand{\C}{\mathbb{C}} 
\newcommand{\Z}{\mathbb{Z}}

\renewcommand{\H}{{\mathbb{H}^2}}
\newcommand{\s}{\small}
\renewcommand{\ss}{\scriptsize}

\newcommand{\dist}{\mathrm{dist}}  

\newcommand{\SL}{\mathrm{SL}}

\newcommand{\GL}{\mathrm{GL}}

\newcommand{\SLt}{{\SL_2(\R)}}

\newcommand{\SO}{\mathrm{SO}}

\newcommand{\slt}{\mathfrak{sl}_2(\R)}

\newcommand{\n}{\noindent}

\newcommand{\mn}{\medskip\noindent}

\newcommand{\be}{\begin{equation}}
\newcommand{\ee}{\end{equation}}

\renewcommand{\tG}{\widetilde\Gamma}


%% file: figures/no_skid.pdf_tex
\begingroup%
  \makeatletter%
  \providecommand\color[2][]{%
    \errmessage{(Inkscape) Color is used for the text in Inkscape, but the package 'color.sty' is not loaded}%
    \renewcommand\color[2][]{}%
  }%
  \providecommand\transparent[1]{%
    \errmessage{(Inkscape) Transparency is used (non-zero) for the text in Inkscape, but the package 'transparent.sty' is not loaded}%
    \renewcommand\transparent[1]{}%
  }%
  \providecommand\rotatebox[2]{#2}%
  \newcommand*\fsize{\dimexpr\f@size pt\relax}%
  \newcommand*\lineheight[1]{\fontsize{\fsize}{#1\fsize}\selectfont}%
  \ifx\svgwidth\undefined%
    \setlength{\unitlength}{673.45207214bp}%
    \ifx\svgscale\undefined%
      \relax%
    \else%
      \setlength{\unitlength}{\unitlength * \real{\svgscale}}%
    \fi%
  \else%
    \setlength{\unitlength}{\svgwidth}%
  \fi%
  \global\let\svgwidth\undefined%
  \global\let\svgscale\undefined%
  \makeatother%
  \begin{picture}(1,0.2089004)%
    \lineheight{1}%
    \setlength\tabcolsep{0pt}%
    \put(0,0){\includegraphics[width=\unitlength,page=1]{no_skid.pdf}}%
    \put(0.03722178,0.1858713){\color[rgb]{0.03921569,0,0}\rotatebox{0.40487591}{\makebox(0,0)[lt]{\lineheight{1.25}\smash{\begin{tabular}[t]{l}$\G$\end{tabular}}}}}%
    \put(0.13381193,0.14162326){\color[rgb]{0.03921569,0,0}\rotatebox{0.40487591}{\makebox(0,0)[lt]{\lineheight{1.25}\smash{\begin{tabular}[t]{l}$\g$\end{tabular}}}}}%
    \put(0,0){\includegraphics[width=\unitlength,page=2]{no_skid.pdf}}%
    \put(0.77743199,0.17989845){\color[rgb]{1,0,0.0745098}\rotatebox{15.176442}{\makebox(0,0)[lt]{\lineheight{1.25}\smash{\begin{tabular}[t]{l}     \end{tabular}}}}}%
    \put(0,0){\includegraphics[width=\unitlength,page=3]{no_skid.pdf}}%
    \put(0.94797922,0.00846375){\color[rgb]{0.03921569,0,0}\rotatebox{0.40487597}{\makebox(0,0)[lt]{\lineheight{1.25}\smash{\begin{tabular}[t]{l}\s$b(\G)$ \end{tabular}}}}}%
    \put(0,0){\includegraphics[width=\unitlength,page=4]{no_skid.pdf}}%
  \end{picture}%
\endgroup%

%% file: figures/stereo.pdf_tex
\begingroup%
  \makeatletter%
  \providecommand\color[2][]{%
    \errmessage{(Inkscape) Color is used for the text in Inkscape, but the package 'color.sty' is not loaded}%
    \renewcommand\color[2][]{}%
  }%
  \providecommand\transparent[1]{%
    \errmessage{(Inkscape) Transparency is used (non-zero) for the text in Inkscape, but the package 'transparent.sty' is not loaded}%
    \renewcommand\transparent[1]{}%
  }%
  \providecommand\rotatebox[2]{#2}%
  \newcommand*\fsize{\dimexpr\f@size pt\relax}%
  \newcommand*\lineheight[1]{\fontsize{\fsize}{#1\fsize}\selectfont}%
  \ifx\svgwidth\undefined%
    \setlength{\unitlength}{188.55212114bp}%
    \ifx\svgscale\undefined%
      \relax%
    \else%
      \setlength{\unitlength}{\unitlength * \real{\svgscale}}%
    \fi%
  \else%
    \setlength{\unitlength}{\svgwidth}%
  \fi%
  \global\let\svgwidth\undefined%
  \global\let\svgscale\undefined%
  \makeatother%
  \begin{picture}(1,1.27000685)%
    \lineheight{1}%
    \setlength\tabcolsep{0pt}%
    \put(0,0){\includegraphics[width=\unitlength,page=1]{stereo.pdf}}%
    \put(0.60700216,0.61410694){\color[rgb]{0,0,0}\makebox(0,0)[lt]{\lineheight{1.25}\smash{\begin{tabular}[t]{l}\s$\theta$\end{tabular}}}}%
    \put(0.3659156,0.8583664){\color[rgb]{0,0,0}\makebox(0,0)[lt]{\lineheight{1.25}\smash{\begin{tabular}[t]{l}\s$p$\end{tabular}}}}%
    \put(0,0){\includegraphics[width=\unitlength,page=2]{stereo.pdf}}%
    \put(0.54686807,1.1898657){\color[rgb]{0,0,0}\makebox(0,0)[lt]{\lineheight{1.25}\smash{\begin{tabular}[t]{l}\s$\RP^1$\end{tabular}}}}%
    \put(0.82932772,1.02697204){\color[rgb]{0,0,0}\makebox(0,0)[lt]{\lineheight{1.25}\smash{\begin{tabular}[t]{l}\s$e^{i\theta}$\end{tabular}}}}%
  \end{picture}%
\endgroup%

%% file: figures/monodromies.pdf_tex
\begingroup%
  \makeatletter%
  \providecommand\color[2][]{%
    \errmessage{(Inkscape) Color is used for the text in Inkscape, but the package 'color.sty' is not loaded}%
    \renewcommand\color[2][]{}%
  }%
  \providecommand\transparent[1]{%
    \errmessage{(Inkscape) Transparency is used (non-zero) for the text in Inkscape, but the package 'transparent.sty' is not loaded}%
    \renewcommand\transparent[1]{}%
  }%
  \providecommand\rotatebox[2]{#2}%
  \newcommand*\fsize{\dimexpr\f@size pt\relax}%
  \newcommand*\lineheight[1]{\fontsize{\fsize}{#1\fsize}\selectfont}%
  \ifx\svgwidth\undefined%
    \setlength{\unitlength}{646.69299316bp}%
    \ifx\svgscale\undefined%
      \relax%
    \else%
      \setlength{\unitlength}{\unitlength * \real{\svgscale}}%
    \fi%
  \else%
    \setlength{\unitlength}{\svgwidth}%
  \fi%
  \global\let\svgwidth\undefined%
  \global\let\svgscale\undefined%
  \makeatother%
  \begin{picture}(1,0.25788559)%
    \lineheight{1}%
    \setlength\tabcolsep{0pt}%
    \put(0,0){\includegraphics[width=\unitlength,page=1]{monodromies.pdf}}%
    \put(-0.00001416,0.00376435){\color[rgb]{0.04313725,0.02352941,0.02352941}\makebox(0,0)[lt]{\lineheight{1.25}\smash{\begin{tabular}[t]{l}\s(a)\end{tabular}}}}%
    \put(0.46166044,0.00376435){\color[rgb]{0.04313725,0.02352941,0.02352941}\makebox(0,0)[lt]{\lineheight{1.25}\smash{\begin{tabular}[t]{l}\s(b)\end{tabular}}}}%
    \put(0.78498518,0.00206783){\color[rgb]{0.04313725,0.02352941,0.02352941}\makebox(0,0)[lt]{\lineheight{1.25}\smash{\begin{tabular}[t]{l}\s(c)\end{tabular}}}}%
  \end{picture}%
\endgroup%

%% file: figures/rectangles.pdf_tex
\begingroup%
  \makeatletter%
  \providecommand\color[2][]{%
    \errmessage{(Inkscape) Color is used for the text in Inkscape, but the package 'color.sty' is not loaded}%
    \renewcommand\color[2][]{}%
  }%
  \providecommand\transparent[1]{%
    \errmessage{(Inkscape) Transparency is used (non-zero) for the text in Inkscape, but the package 'transparent.sty' is not loaded}%
    \renewcommand\transparent[1]{}%
  }%
  \providecommand\rotatebox[2]{#2}%
  \newcommand*\fsize{\dimexpr\f@size pt\relax}%
  \newcommand*\lineheight[1]{\fontsize{\fsize}{#1\fsize}\selectfont}%
  \ifx\svgwidth\undefined%
    \setlength{\unitlength}{481.33955383bp}%
    \ifx\svgscale\undefined%
      \relax%
    \else%
      \setlength{\unitlength}{\unitlength * \real{\svgscale}}%
    \fi%
  \else%
    \setlength{\unitlength}{\svgwidth}%
  \fi%
  \global\let\svgwidth\undefined%
  \global\let\svgscale\undefined%
  \makeatother%
  \begin{picture}(1,0.98173489)%
    \lineheight{1}%
    \setlength\tabcolsep{0pt}%
    \put(0,0){\includegraphics[width=\unitlength,page=1]{rectangles.pdf}}%
    \put(0.03291607,-0.01214819){\makebox(0,0)[lt]{\lineheight{1.25}\smash{\begin{tabular}[t]{l}\ss0\end{tabular}}}}%
    \put(0,0){\includegraphics[width=\unitlength,page=2]{rectangles.pdf}}%
    \put(0.25111962,-0.01214819){\makebox(0,0)[lt]{\lineheight{1.25}\smash{\begin{tabular}[t]{l}\ss2\end{tabular}}}}%
    \put(0,0){\includegraphics[width=\unitlength,page=3]{rectangles.pdf}}%
    \put(0.46932316,-0.01214819){\makebox(0,0)[lt]{\lineheight{1.25}\smash{\begin{tabular}[t]{l}\ss4\end{tabular}}}}%
    \put(0,0){\includegraphics[width=\unitlength,page=4]{rectangles.pdf}}%
    \put(0.68752672,-0.01214819){\makebox(0,0)[lt]{\lineheight{1.25}\smash{\begin{tabular}[t]{l}\ss6\end{tabular}}}}%
    \put(0,0){\includegraphics[width=\unitlength,page=5]{rectangles.pdf}}%
    \put(0.90573024,-0.01214819){\makebox(0,0)[lt]{\lineheight{1.25}\smash{\begin{tabular}[t]{l}\ss8\end{tabular}}}}%
    \put(0,0){\includegraphics[width=\unitlength,page=6]{rectangles.pdf}}%
    \put(-0.0096786,0.0418729){\makebox(0,0)[lt]{\lineheight{1.25}\smash{\begin{tabular}[t]{l}\ss0\end{tabular}}}}%
    \put(0,0){\includegraphics[width=\unitlength,page=7]{rectangles.pdf}}%
    \put(-0.0096786,0.26007646){\makebox(0,0)[lt]{\lineheight{1.25}\smash{\begin{tabular}[t]{l}\ss2\end{tabular}}}}%
    \put(0,0){\includegraphics[width=\unitlength,page=8]{rectangles.pdf}}%
    \put(-0.0096786,0.47828003){\makebox(0,0)[lt]{\lineheight{1.25}\smash{\begin{tabular}[t]{l}\ss4\end{tabular}}}}%
    \put(0,0){\includegraphics[width=\unitlength,page=9]{rectangles.pdf}}%
    \put(-0.0096786,0.69648357){\makebox(0,0)[lt]{\lineheight{1.25}\smash{\begin{tabular}[t]{l}\ss6\end{tabular}}}}%
    \put(0,0){\includegraphics[width=\unitlength,page=10]{rectangles.pdf}}%
    \put(-0.0096786,0.91468713){\makebox(0,0)[lt]{\lineheight{1.25}\smash{\begin{tabular}[t]{l}\ss8\end{tabular}}}}%
    \put(0.70877019,0.85677222){\color[rgb]{0,0,0}\makebox(0,0)[lt]{\lineheight{1.25}\smash{\begin{tabular}[t]{l}\ss Parabolic monodromy\end{tabular}}}}%
    \put(0.70877019,0.80379505){\color[rgb]{0,0,0}\makebox(0,0)[lt]{\lineheight{1.25}\smash{\begin{tabular}[t]{l}\ss Area = $\pi$\end{tabular}}}}%
    \put(0.70877019,0.75086659){\color[rgb]{0,0,0}\makebox(0,0)[lt]{\lineheight{1.25}\smash{\begin{tabular}[t]{l}\ss Constant perimeter\\\end{tabular}}}}%
    \put(0,0){\includegraphics[width=\unitlength,page=11]{rectangles.pdf}}%
    \put(0.24234976,0.39104459){\color[rgb]{0,0,0}\rotatebox{-46.945316}{\makebox(0,0)[lt]{\lineheight{1.25}\smash{\begin{tabular}[t]{l}\ss Parabolic\end{tabular}}}}}%
    \put(0.49865546,0.37091483){\color[rgb]{0,0,0}\rotatebox{1.6275888}{\makebox(0,0)[lt]{\lineheight{1.25}\smash{\begin{tabular}[t]{l}\ss Hyperbolic\end{tabular}}}}}%
    \put(0.078443,0.11549159){\color[rgb]{0,0,0}\rotatebox{1.6275888}{\makebox(0,0)[lt]{\lineheight{1.25}\smash{\begin{tabular}[t]{l}\ss Eliptic\end{tabular}}}}}%
    \put(0.96484355,0.04069949){\color[rgb]{0,0,0}\rotatebox{1.6239585}{\makebox(0,0)[lt]{\lineheight{1.25}\smash{\begin{tabular}[t]{l}\s$a$\end{tabular}}}}}%
    \put(0.03426653,0.97408617){\color[rgb]{0,0,0}\rotatebox{1.6239585}{\makebox(0,0)[lt]{\lineheight{1.25}\smash{\begin{tabular}[t]{l}\s$b$\end{tabular}}}}}%
    \put(0,0){\includegraphics[width=\unitlength,page=12]{rectangles.pdf}}%
  \end{picture}%
\endgroup%

%% file: figures/ellipses2.pdf_tex
\begingroup%
  \makeatletter%
  \providecommand\color[2][]{%
    \errmessage{(Inkscape) Color is used for the text in Inkscape, but the package 'color.sty' is not loaded}%
    \renewcommand\color[2][]{}%
  }%
  \providecommand\transparent[1]{%
    \errmessage{(Inkscape) Transparency is used (non-zero) for the text in Inkscape, but the package 'transparent.sty' is not loaded}%
    \renewcommand\transparent[1]{}%
  }%
  \providecommand\rotatebox[2]{#2}%
  \newcommand*\fsize{\dimexpr\f@size pt\relax}%
  \newcommand*\lineheight[1]{\fontsize{\fsize}{#1\fsize}\selectfont}%
  \ifx\svgwidth\undefined%
    \setlength{\unitlength}{779.80603027bp}%
    \ifx\svgscale\undefined%
      \relax%
    \else%
      \setlength{\unitlength}{\unitlength * \real{\svgscale}}%
    \fi%
  \else%
    \setlength{\unitlength}{\svgwidth}%
  \fi%
  \global\let\svgwidth\undefined%
  \global\let\svgscale\undefined%
  \makeatother%
  \begin{picture}(1,0.50051548)%
    \lineheight{1}%
    \setlength\tabcolsep{0pt}%
    \put(0.71047128,0.02631943){\color[rgb]{0,0,0}\makebox(0,0)[lt]{\lineheight{1.25}\smash{\begin{tabular}[t]{l}\s$a$\end{tabular}}}}%
    \put(0.24245195,0.48947007){\color[rgb]{0,0,0}\makebox(0,0)[lt]{\lineheight{1.25}\smash{\begin{tabular}[t]{l}\s$b$\end{tabular}}}}%
    \put(0.29644424,0.06931507){\color[rgb]{0,0,0}\makebox(0,0)[lt]{\lineheight{1.25}\smash{\begin{tabular}[t]{l}\ss Elliptic\end{tabular}}}}%
    \put(0.49766725,0.15379645){\color[rgb]{0,0,0}\makebox(0,0)[lt]{\lineheight{1.25}\smash{\begin{tabular}[t]{l}\ss Hyperbolic\end{tabular}}}}%
    \put(0,0){\includegraphics[width=\unitlength,page=1]{ellipses2.pdf}}%
    \put(0.8428018,0.29033105){\color[rgb]{0,0,0}\makebox(0,0)[lt]{\lineheight{1.25}\smash{\begin{tabular}[t]{l}\ss Parabolic monodromy\end{tabular}}}}%
    \put(0.8428018,0.25763059){\color[rgb]{0,0,0}\makebox(0,0)[lt]{\lineheight{1.25}\smash{\begin{tabular}[t]{l}\ss Area = $\pi$\end{tabular}}}}%
    \put(0,0){\includegraphics[width=\unitlength,page=2]{ellipses2.pdf}}%
    \put(0.8428018,0.22493015){\color[rgb]{0,0,0}\makebox(0,0)[lt]{\lineheight{1.25}\smash{\begin{tabular}[t]{l}\ss Perimeter = $2\pi$\end{tabular}}}}%
    \put(0,0){\includegraphics[width=\unitlength,page=3]{ellipses2.pdf}}%
    \put(0.8428018,0.19030617){\color[rgb]{0,0,0}\makebox(0,0)[lt]{\lineheight{1.25}\smash{\begin{tabular}[t]{l}\ss $|\kappa|\leq 1$\end{tabular}}}}%
    \put(0,0){\includegraphics[width=\unitlength,page=4]{ellipses2.pdf}}%
    \put(0.36935545,0.00542604){\color[rgb]{0,0,0}\makebox(0,0)[lt]{\lineheight{1.25}\smash{\begin{tabular}[t]{l}\ss$1$\end{tabular}}}}%
    \put(0.49438654,0.00542604){\color[rgb]{0,0,0}\makebox(0,0)[lt]{\lineheight{1.25}\smash{\begin{tabular}[t]{l}\ss$2$\end{tabular}}}}%
    \put(0.61557052,0.00542604){\color[rgb]{0,0,0}\makebox(0,0)[lt]{\lineheight{1.25}\smash{\begin{tabular}[t]{l}\ss$3$\end{tabular}}}}%
    \put(0,0){\includegraphics[width=\unitlength,page=5]{ellipses2.pdf}}%
    \put(0.22949579,0.15181489){\color[rgb]{0,0,0}\makebox(0,0)[lt]{\lineheight{1.25}\smash{\begin{tabular}[t]{l}\ss$1$\end{tabular}}}}%
    \put(0.22555284,0.27348685){\color[rgb]{0,0,0}\makebox(0,0)[lt]{\lineheight{1.25}\smash{\begin{tabular}[t]{l}\ss$2$\end{tabular}}}}%
    \put(0.22720844,0.39577456){\color[rgb]{0,0,0}\makebox(0,0)[lt]{\lineheight{1.25}\smash{\begin{tabular}[t]{l}\ss$3$\end{tabular}}}}%
  \end{picture}%
\endgroup%

%% file: figures/dev_broken.pdf_tex
\begingroup%
  \makeatletter%
  \providecommand\color[2][]{%
    \errmessage{(Inkscape) Color is used for the text in Inkscape, but the package 'color.sty' is not loaded}%
    \renewcommand\color[2][]{}%
  }%
  \providecommand\transparent[1]{%
    \errmessage{(Inkscape) Transparency is used (non-zero) for the text in Inkscape, but the package 'transparent.sty' is not loaded}%
    \renewcommand\transparent[1]{}%
  }%
  \providecommand\rotatebox[2]{#2}%
  \newcommand*\fsize{\dimexpr\f@size pt\relax}%
  \newcommand*\lineheight[1]{\fontsize{\fsize}{#1\fsize}\selectfont}%
  \ifx\svgwidth\undefined%
    \setlength{\unitlength}{38.1677177bp}%
    \ifx\svgscale\undefined%
      \relax%
    \else%
      \setlength{\unitlength}{\unitlength * \real{\svgscale}}%
    \fi%
  \else%
    \setlength{\unitlength}{\svgwidth}%
  \fi%
  \global\let\svgwidth\undefined%
  \global\let\svgscale\undefined%
  \makeatother%
  \begin{picture}(1,0.29537898)%
    \lineheight{1}%
    \setlength\tabcolsep{0pt}%
    \put(0,0){\includegraphics[width=\unitlength,page=1]{dev_broken.pdf}}%
    \put(0.08690351,0.10614469){\color[rgb]{0,0,0}\makebox(0,0)[lt]{\lineheight{1.25}\smash{\begin{tabular}[t]{l}\s$\G_1$\end{tabular}}}}%
    \put(0,0){\includegraphics[width=\unitlength,page=2]{dev_broken.pdf}}%
    \put(0.40939967,0.0324409){\color[rgb]{0,0,0}\makebox(0,0)[lt]{\lineheight{1.25}\smash{\begin{tabular}[t]{l} \end{tabular}}}}%
    \put(0.0483322,0.01018969){\color[rgb]{0,0,0}\makebox(0,0)[lt]{\lineheight{1.25}\smash{\begin{tabular}[t]{l} \end{tabular}}}}%
    \put(0.19491323,0.1965298){\color[rgb]{0,0,0}\makebox(0,0)[lt]{\lineheight{1.25}\smash{\begin{tabular}[t]{l}\s$\G_2$\end{tabular}}}}%
    \put(0,0){\includegraphics[width=\unitlength,page=3]{dev_broken.pdf}}%
    \put(0.28609161,0.18477128){\color[rgb]{0,0,0}\makebox(0,0)[lt]{\lineheight{1.25}\smash{\begin{tabular}[t]{l}\s$\rho$\end{tabular}}}}%
    \put(0.71462183,0.15838431){\color[rgb]{0,0,0}\makebox(0,0)[lt]{\lineheight{1.25}\smash{\begin{tabular}[t]{l}\s$\tG_1$\end{tabular}}}}%
    \put(0.67090999,0.03436741){\color[rgb]{0,0,0}\makebox(0,0)[lt]{\lineheight{1.25}\smash{\begin{tabular}[t]{l} \end{tabular}}}}%
    \put(0.85395606,0.2014781){\color[rgb]{0,0,0}\makebox(0,0)[lt]{\lineheight{1.25}\smash{\begin{tabular}[t]{l}\s$\tG_2$\end{tabular}}}}%
    \put(0,0){\includegraphics[width=\unitlength,page=4]{dev_broken.pdf}}%
    \put(0.90511173,0.13274074){\color[rgb]{0,0,0}\makebox(0,0)[lt]{\lineheight{1.25}\smash{\begin{tabular}[t]{l}\s$\rho$\end{tabular}}}}%
    \put(0,0){\includegraphics[width=\unitlength,page=5]{dev_broken.pdf}}%
    \put(0.17022102,0.00767842){\color[rgb]{0,0,0}\makebox(0,0)[lt]{\lineheight{1.25}\smash{\begin{tabular}[t]{l}\s$\R^2$\end{tabular}}}}%
    \put(0.76098658,0.00894142){\color[rgb]{0,0,0}\makebox(0,0)[lt]{\lineheight{1.25}\smash{\begin{tabular}[t]{l}\s$\H$\end{tabular}}}}%
  \end{picture}%
\endgroup%

%% file: figures/horocycles.pdf_tex
\begingroup%
  \makeatletter%
  \providecommand\color[2][]{%
    \errmessage{(Inkscape) Color is used for the text in Inkscape, but the package 'color.sty' is not loaded}%
    \renewcommand\color[2][]{}%
  }%
  \providecommand\transparent[1]{%
    \errmessage{(Inkscape) Transparency is used (non-zero) for the text in Inkscape, but the package 'transparent.sty' is not loaded}%
    \renewcommand\transparent[1]{}%
  }%
  \providecommand\rotatebox[2]{#2}%
  \newcommand*\fsize{\dimexpr\f@size pt\relax}%
  \newcommand*\lineheight[1]{\fontsize{\fsize}{#1\fsize}\selectfont}%
  \ifx\svgwidth\undefined%
    \setlength{\unitlength}{352.36037127bp}%
    \ifx\svgscale\undefined%
      \relax%
    \else%
      \setlength{\unitlength}{\unitlength * \real{\svgscale}}%
    \fi%
  \else%
    \setlength{\unitlength}{\svgwidth}%
  \fi%
  \global\let\svgwidth\undefined%
  \global\let\svgscale\undefined%
  \makeatother%
  \begin{picture}(1,0.8954832)%
    \lineheight{1}%
    \setlength\tabcolsep{0pt}%
    \put(0,0){\includegraphics[width=\unitlength,page=1]{horocycles.pdf}}%
    \put(0.6809029,0.20584532){\color[rgb]{0,0,0}\makebox(0,0)[lt]{\lineheight{1.25}\smash{\begin{tabular}[t]{l}\s$\beta$\end{tabular}}}}%
    \put(0.315141,0.27899771){\color[rgb]{0,0,0}\makebox(0,0)[lt]{\lineheight{1.25}\smash{\begin{tabular}[t]{l}\s$\g$\end{tabular}}}}%
    \put(0.17696428,0.00101871){\color[rgb]{0,0,0}\makebox(0,0)[lt]{\lineheight{1.25}\smash{\begin{tabular}[t]{l}\ss$p_1$\end{tabular}}}}%
    \put(0.89385765,0.58136089){\color[rgb]{0,0,0}\makebox(0,0)[lt]{\lineheight{1.25}\smash{\begin{tabular}[t]{l}\ss$p_2$\end{tabular}}}}%
    \put(0,0){\includegraphics[width=\unitlength,page=2]{horocycles.pdf}}%
  \end{picture}%
\endgroup%

%% file: figures/ps.pdf_tex
\begingroup%
  \makeatletter%
  \providecommand\color[2][]{%
    \errmessage{(Inkscape) Color is used for the text in Inkscape, but the package 'color.sty' is not loaded}%
    \renewcommand\color[2][]{}%
  }%
  \providecommand\transparent[1]{%
    \errmessage{(Inkscape) Transparency is used (non-zero) for the text in Inkscape, but the package 'transparent.sty' is not loaded}%
    \renewcommand\transparent[1]{}%
  }%
  \providecommand\rotatebox[2]{#2}%
  \newcommand*\fsize{\dimexpr\f@size pt\relax}%
  \newcommand*\lineheight[1]{\fontsize{\fsize}{#1\fsize}\selectfont}%
  \ifx\svgwidth\undefined%
    \setlength{\unitlength}{245.06763128bp}%
    \ifx\svgscale\undefined%
      \relax%
    \else%
      \setlength{\unitlength}{\unitlength * \real{\svgscale}}%
    \fi%
  \else%
    \setlength{\unitlength}{\svgwidth}%
  \fi%
  \global\let\svgwidth\undefined%
  \global\let\svgscale\undefined%
  \makeatother%
  \begin{picture}(1,0.45790709)%
    \lineheight{1}%
    \setlength\tabcolsep{0pt}%
    \put(0,0){\includegraphics[width=\unitlength,page=1]{ps.pdf}}%
    \put(0.56465618,0.19175445){\color[rgb]{0,0,0}\makebox(0,0)[lt]{\lineheight{1.25}\smash{\begin{tabular}[t]{l}\s$\beta$\end{tabular}}}}%
    \put(0.22990956,0.2912928){\color[rgb]{0,0,0}\makebox(0,0)[lt]{\lineheight{1.25}\smash{\begin{tabular}[t]{l}\s$\g$\end{tabular}}}}%
    \put(0,0){\includegraphics[width=\unitlength,page=2]{ps.pdf}}%
    \put(0.84738763,0.01350489){\color[rgb]{0,0,0}\makebox(0,0)[lt]{\lineheight{1.25}\smash{\begin{tabular}[t]{l}\s$p_2$\end{tabular}}}}%
    \put(0,0){\includegraphics[width=\unitlength,page=3]{ps.pdf}}%
    \put(0.08311541,0.01323485){\color[rgb]{0,0,0}\makebox(0,0)[lt]{\lineheight{1.25}\smash{\begin{tabular}[t]{l}\s$p_1$\end{tabular}}}}%
  \end{picture}%
\endgroup%

%% file: figures/ps1.pdf_tex
\begingroup%
  \makeatletter%
  \providecommand\color[2][]{%
    \errmessage{(Inkscape) Color is used for the text in Inkscape, but the package 'color.sty' is not loaded}%
    \renewcommand\color[2][]{}%
  }%
  \providecommand\transparent[1]{%
    \errmessage{(Inkscape) Transparency is used (non-zero) for the text in Inkscape, but the package 'transparent.sty' is not loaded}%
    \renewcommand\transparent[1]{}%
  }%
  \providecommand\rotatebox[2]{#2}%
  \newcommand*\fsize{\dimexpr\f@size pt\relax}%
  \newcommand*\lineheight[1]{\fontsize{\fsize}{#1\fsize}\selectfont}%
  \ifx\svgwidth\undefined%
    \setlength{\unitlength}{169.10649337bp}%
    \ifx\svgscale\undefined%
      \relax%
    \else%
      \setlength{\unitlength}{\unitlength * \real{\svgscale}}%
    \fi%
  \else%
    \setlength{\unitlength}{\svgwidth}%
  \fi%
  \global\let\svgwidth\undefined%
  \global\let\svgscale\undefined%
  \makeatother%
  \begin{picture}(1,0.67172362)%
    \lineheight{1}%
    \setlength\tabcolsep{0pt}%
    \put(0.04097204,0.32491936){\color[rgb]{0,0,0}\makebox(0,0)[lt]{\lineheight{1.25}\smash{\begin{tabular}[t]{l}\s$\beta$\end{tabular}}}}%
    \put(0.48394554,0.61036406){\color[rgb]{0,0,0}\makebox(0,0)[lt]{\lineheight{1.25}\smash{\begin{tabular}[t]{l}\s$\g$\end{tabular}}}}%
    \put(0,0){\includegraphics[width=\unitlength,page=1]{ps1.pdf}}%
    \put(0.37671188,0.0153149){\color[rgb]{0,0,0}\makebox(0,0)[lt]{\lineheight{1.25}\smash{\begin{tabular}[t]{l}\s$p_1$\end{tabular}}}}%
    \put(0,0){\includegraphics[width=\unitlength,page=2]{ps1.pdf}}%
  \end{picture}%
\endgroup%

%% file: figures/tG.pdf_tex
\begingroup%
  \makeatletter%
  \providecommand\color[2][]{%
    \errmessage{(Inkscape) Color is used for the text in Inkscape, but the package 'color.sty' is not loaded}%
    \renewcommand\color[2][]{}%
  }%
  \providecommand\transparent[1]{%
    \errmessage{(Inkscape) Transparency is used (non-zero) for the text in Inkscape, but the package 'transparent.sty' is not loaded}%
    \renewcommand\transparent[1]{}%
  }%
  \providecommand\rotatebox[2]{#2}%
  \newcommand*\fsize{\dimexpr\f@size pt\relax}%
  \newcommand*\lineheight[1]{\fontsize{\fsize}{#1\fsize}\selectfont}%
  \ifx\svgwidth\undefined%
    \setlength{\unitlength}{187.23957933bp}%
    \ifx\svgscale\undefined%
      \relax%
    \else%
      \setlength{\unitlength}{\unitlength * \real{\svgscale}}%
    \fi%
  \else%
    \setlength{\unitlength}{\svgwidth}%
  \fi%
  \global\let\svgwidth\undefined%
  \global\let\svgscale\undefined%
  \makeatother%
  \begin{picture}(1,0.39292137)%
    \lineheight{1}%
    \setlength\tabcolsep{0pt}%
    \put(0,0){\includegraphics[width=\unitlength,page=1]{tG.pdf}}%
    \put(0.79455835,0.00280166){\color[rgb]{0,0,0}\makebox(0,0)[lt]{\lineheight{1.25}\smash{\begin{tabular}[t]{l}\s$x$\end{tabular}}}}%
    \put(0.35302724,0.1173317){\color[rgb]{0,0,0}\makebox(0,0)[lt]{\lineheight{1.25}\smash{\begin{tabular}[t]{l}\ss$\tG(0)$\end{tabular}}}}%
    \put(0.1675325,0.26420207){\color[rgb]{1,1,1}\makebox(0,0)[lt]{\lineheight{1.25}\smash{\begin{tabular}[t]{l} \end{tabular}}}}%
    \put(0,0){\includegraphics[width=\unitlength,page=2]{tG.pdf}}%
    \put(0.03729287,0.35499278){\color[rgb]{0,0,0}\makebox(0,0)[lt]{\lineheight{1.25}\smash{\begin{tabular}[t]{l}\s$y$\end{tabular}}}}%
    \put(0,0){\includegraphics[width=\unitlength,page=3]{tG.pdf}}%
    \put(0.73324456,0.38311393){\color[rgb]{0,0,0}\makebox(0,0)[lt]{\lineheight{1.25}\smash{\begin{tabular}[t]{l}\ss$\tG$\end{tabular}}}}%
    \put(0,0){\includegraphics[width=\unitlength,page=4]{tG.pdf}}%
    \put(0.64507623,0.26905687){\color[rgb]{0,0,0}\makebox(0,0)[lt]{\lineheight{1.25}\smash{\begin{tabular}[t]{l}\ss$\tG(L)$\end{tabular}}}}%
    \put(0.17849946,0.04190768){\color[rgb]{0,0,0}\makebox(0,0)[lt]{\lineheight{1.25}\smash{\begin{tabular}[t]{l}\ss$\tG(-L)$\end{tabular}}}}%
    \put(0,0){\includegraphics[width=\unitlength,page=5]{tG.pdf}}%
  \end{picture}%
\endgroup%

%% file: figures/bridgeman.pdf_tex
\begingroup%
  \makeatletter%
  \providecommand\color[2][]{%
    \errmessage{(Inkscape) Color is used for the text in Inkscape, but the package 'color.sty' is not loaded}%
    \renewcommand\color[2][]{}%
  }%
  \providecommand\transparent[1]{%
    \errmessage{(Inkscape) Transparency is used (non-zero) for the text in Inkscape, but the package 'transparent.sty' is not loaded}%
    \renewcommand\transparent[1]{}%
  }%
  \providecommand\rotatebox[2]{#2}%
  \newcommand*\fsize{\dimexpr\f@size pt\relax}%
  \newcommand*\lineheight[1]{\fontsize{\fsize}{#1\fsize}\selectfont}%
  \ifx\svgwidth\undefined%
    \setlength{\unitlength}{447.67003337bp}%
    \ifx\svgscale\undefined%
      \relax%
    \else%
      \setlength{\unitlength}{\unitlength * \real{\svgscale}}%
    \fi%
  \else%
    \setlength{\unitlength}{\svgwidth}%
  \fi%
  \global\let\svgwidth\undefined%
  \global\let\svgscale\undefined%
  \makeatother%
  \begin{picture}(1,0.45584027)%
    \lineheight{1}%
    \setlength\tabcolsep{0pt}%
    \put(0,0){\includegraphics[width=\unitlength,page=1]{bridgeman.pdf}}%
    \put(0.62089636,0.03933087){\color[rgb]{0,0,0}\makebox(0,0)[lt]{\lineheight{1.25}\smash{\begin{tabular}[t]{l}\s$\bar\g$\end{tabular}}}}%
    \put(0,0){\includegraphics[width=\unitlength,page=2]{bridgeman.pdf}}%
    \put(0.59878099,0.40402635){\color[rgb]{0,0,0}\makebox(0,0)[lt]{\lineheight{1.25}\smash{\begin{tabular}[t]{l}\s$\g$\end{tabular}}}}%
    \put(0.39590812,0.19350713){\color[rgb]{0,0,0}\makebox(0,0)[lt]{\lineheight{1.25}\smash{\begin{tabular}[t]{l}\s$g$\end{tabular}}}}%
    \put(0.24562281,0.26449245){\color[rgb]{0,0,0}\makebox(0,0)[lt]{\lineheight{1.25}\smash{\begin{tabular}[t]{l}\s$\theta_2$\end{tabular}}}}%
    \put(0.64674619,0.26590105){\color[rgb]{0,0,0}\makebox(0,0)[lt]{\lineheight{1.25}\smash{\begin{tabular}[t]{l}\s$\theta_1$\end{tabular}}}}%
    \put(0.85542342,0.21838411){\color[rgb]{0,0,0}\makebox(0,0)[lt]{\lineheight{1.25}\smash{\begin{tabular}[t]{l}\ss$\tG(-nL)$\end{tabular}}}}%
    \put(0.00717246,0.22265649){\color[rgb]{0,0,0}\makebox(0,0)[lt]{\lineheight{1.25}\smash{\begin{tabular}[t]{l}\ss$\tG(nL)$\end{tabular}}}}%
    \put(0,0){\includegraphics[width=\unitlength,page=3]{bridgeman.pdf}}%
  \end{picture}%
\endgroup%

%% file: figures/hyperbol.pdf_tex
\begingroup%
  \makeatletter%
  \providecommand\color[2][]{%
    \errmessage{(Inkscape) Color is used for the text in Inkscape, but the package 'color.sty' is not loaded}%
    \renewcommand\color[2][]{}%
  }%
  \providecommand\transparent[1]{%
    \errmessage{(Inkscape) Transparency is used (non-zero) for the text in Inkscape, but the package 'transparent.sty' is not loaded}%
    \renewcommand\transparent[1]{}%
  }%
  \providecommand\rotatebox[2]{#2}%
  \newcommand*\fsize{\dimexpr\f@size pt\relax}%
  \newcommand*\lineheight[1]{\fontsize{\fsize}{#1\fsize}\selectfont}%
  \ifx\svgwidth\undefined%
    \setlength{\unitlength}{187.23957933bp}%
    \ifx\svgscale\undefined%
      \relax%
    \else%
      \setlength{\unitlength}{\unitlength * \real{\svgscale}}%
    \fi%
  \else%
    \setlength{\unitlength}{\svgwidth}%
  \fi%
  \global\let\svgwidth\undefined%
  \global\let\svgscale\undefined%
  \makeatother%
  \begin{picture}(1,0.39292137)%
    \lineheight{1}%
    \setlength\tabcolsep{0pt}%
    \put(0,0){\includegraphics[width=\unitlength,page=1]{hyperbol.pdf}}%
    \put(0.95478103,0.00280166){\color[rgb]{0,0,0}\makebox(0,0)[lt]{\lineheight{1.25}\smash{\begin{tabular}[t]{l}\s$x$\end{tabular}}}}%
    \put(0,0){\includegraphics[width=\unitlength,page=2]{hyperbol.pdf}}%
    \put(0.35302724,0.1173317){\color[rgb]{0,0,0}\makebox(0,0)[lt]{\lineheight{1.25}\smash{\begin{tabular}[t]{l}\ss$z_0$\end{tabular}}}}%
    \put(0.1675325,0.26420207){\color[rgb]{1,1,1}\makebox(0,0)[lt]{\lineheight{1.25}\smash{\begin{tabular}[t]{l} \end{tabular}}}}%
    \put(0,0){\includegraphics[width=\unitlength,page=3]{hyperbol.pdf}}%
    \put(0.64985468,0.2577307){\color[rgb]{0,0,0}\makebox(0,0)[lt]{\lineheight{1.25}\smash{\begin{tabular}[t]{l}\ss$\lambda z_0$\end{tabular}}}}%
    \put(0.20842915,0.0403432){\color[rgb]{0,0,0}\makebox(0,0)[lt]{\lineheight{1.25}\smash{\begin{tabular}[t]{l}\ss$z_0/\lambda$\end{tabular}}}}%
    \put(0.22721389,0.13801611){\color[rgb]{0,0,0}\makebox(0,0)[lt]{\lineheight{1.25}\smash{\begin{tabular}[t]{l}\s$\alpha$\end{tabular}}}}%
    \put(0.55330757,0.30724857){\color[rgb]{0,0,0}\makebox(0,0)[lt]{\lineheight{1.25}\smash{\begin{tabular}[t]{l}\s$\alpha$\end{tabular}}}}%
    \put(0.28423179,0.29490286){\color[rgb]{0,0,0}\makebox(0,0)[lt]{\lineheight{1.25}\smash{\begin{tabular}[t]{l}\s$g$\end{tabular}}}}%
    \put(0.03729287,0.35499278){\color[rgb]{0,0,0}\makebox(0,0)[lt]{\lineheight{1.25}\smash{\begin{tabular}[t]{l}\s$y$\end{tabular}}}}%
    \put(0,0){\includegraphics[width=\unitlength,page=4]{hyperbol.pdf}}%
    \put(0.06372479,0.06176012){\color[rgb]{0,0,0}\makebox(0,0)[lt]{\lineheight{1.25}\smash{\begin{tabular}[t]{l}\s$\alpha_\infty$\end{tabular}}}}%
    \put(0,0){\includegraphics[width=\unitlength,page=5]{hyperbol.pdf}}%
    \put(0.79230651,0.33597068){\color[rgb]{0,0,0}\makebox(0,0)[lt]{\lineheight{1.25}\smash{\begin{tabular}[t]{l}\ss$\R^ + z_0$\end{tabular}}}}%
    \put(0,0){\includegraphics[width=\unitlength,page=6]{hyperbol.pdf}}%
    \put(0.71490197,0.38311393){\color[rgb]{0,0,0}\makebox(0,0)[lt]{\lineheight{1.25}\smash{\begin{tabular}[t]{l}\ss$\tG$\end{tabular}}}}%
    \put(0,0){\includegraphics[width=\unitlength,page=7]{hyperbol.pdf}}%
  \end{picture}%
\endgroup%

%% file: figures/arm.pdf_tex
\begingroup%
  \makeatletter%
  \providecommand\color[2][]{%
    \errmessage{(Inkscape) Color is used for the text in Inkscape, but the package 'color.sty' is not loaded}%
    \renewcommand\color[2][]{}%
  }%
  \providecommand\transparent[1]{%
    \errmessage{(Inkscape) Transparency is used (non-zero) for the text in Inkscape, but the package 'transparent.sty' is not loaded}%
    \renewcommand\transparent[1]{}%
  }%
  \providecommand\rotatebox[2]{#2}%
  \newcommand*\fsize{\dimexpr\f@size pt\relax}%
  \newcommand*\lineheight[1]{\fontsize{\fsize}{#1\fsize}\selectfont}%
  \ifx\svgwidth\undefined%
    \setlength{\unitlength}{72.68128643bp}%
    \ifx\svgscale\undefined%
      \relax%
    \else%
      \setlength{\unitlength}{\unitlength * \real{\svgscale}}%
    \fi%
  \else%
    \setlength{\unitlength}{\svgwidth}%
  \fi%
  \global\let\svgwidth\undefined%
  \global\let\svgscale\undefined%
  \makeatother%
  \begin{picture}(1,0.34689715)%
    \lineheight{1}%
    \setlength\tabcolsep{0pt}%
    \put(0,0){\includegraphics[width=\unitlength,page=1]{arm.pdf}}%
    \put(0.0928193,0.02245449){\color[rgb]{0,0,0}\makebox(0,0)[lt]{\lineheight{1.25}\smash{\begin{tabular}[t]{l}\s$q_0$\end{tabular}}}}%
    \put(0.21787825,0.02472186){\color[rgb]{0,0,0}\makebox(0,0)[lt]{\lineheight{1.25}\smash{\begin{tabular}[t]{l}\s$q_1$\end{tabular}}}}%
    \put(0,0){\includegraphics[width=\unitlength,page=2]{arm.pdf}}%
    \put(0.01893895,0.21126431){\color[rgb]{0,0,0}\makebox(0,0)[lt]{\lineheight{1.25}\smash{\begin{tabular}[t]{l}\s$q_n$\end{tabular}}}}%
    \put(0,0){\includegraphics[width=\unitlength,page=3]{arm.pdf}}%
    \put(0.16880099,0.31060311){\color[rgb]{0,0,0}\makebox(0,0)[lt]{\lineheight{1.25}\smash{\begin{tabular}[t]{l}\s$q_\nmo$\end{tabular}}}}%
    \put(0.28845122,0.32074227){\color[rgb]{0,0,0}\makebox(0,0)[lt]{\lineheight{1.25}\smash{\begin{tabular}[t]{l}        \end{tabular}}}}%
    \put(0,0){\includegraphics[width=\unitlength,page=4]{arm.pdf}}%
    \put(0.71305959,0.01238674){\color[rgb]{0,0,0}\makebox(0,0)[lt]{\lineheight{1.25}\smash{\begin{tabular}[t]{l}\s$\tilde q_0$\end{tabular}}}}%
    \put(0.83962491,0.03671448){\color[rgb]{0,0,0}\makebox(0,0)[lt]{\lineheight{1.25}\smash{\begin{tabular}[t]{l}\s$\tilde q_1$\end{tabular}}}}%
    \put(0.70749145,0.28905244){\color[rgb]{0,0,0}\makebox(0,0)[lt]{\lineheight{1.25}\smash{\begin{tabular}[t]{l}\s$\tilde q_n$\end{tabular}}}}%
    \put(0.89106818,0.31798335){\color[rgb]{0,0,0}\makebox(0,0)[lt]{\lineheight{1.25}\smash{\begin{tabular}[t]{l}\s$\tilde q_\nmo$\end{tabular}}}}%
    \put(0.25838298,0.06558737){\color[rgb]{0,0,0}\makebox(0,0)[lt]{\lineheight{1.25}\smash{\begin{tabular}[t]{l}\ss$\alpha_1$\end{tabular}}}}%
    \put(0.8767097,0.07499217){\color[rgb]{0,0,0}\makebox(0,0)[lt]{\lineheight{1.25}\smash{\begin{tabular}[t]{l}\ss$\tilde\alpha_1$\end{tabular}}}}%
    \put(0.8234317,0.17332048){\color[rgb]{0,0,0}\makebox(0,0)[lt]{\lineheight{1.25}\smash{\begin{tabular}[t]{l}\s$\tilde P$\end{tabular}}}}%
    \put(0.16142625,0.15703979){\color[rgb]{0,0,0}\makebox(0,0)[lt]{\lineheight{1.25}\smash{\begin{tabular}[t]{l}\s$  P$\end{tabular}}}}%
    \put(0.05520568,0.11720709){\color[rgb]{0,0,0}\makebox(0,0)[lt]{\lineheight{1.25}\smash{\begin{tabular}[t]{l}\ss$d_n$\end{tabular}}}}%
    \put(0.71993175,0.15687703){\color[rgb]{0,0,0}\makebox(0,0)[lt]{\lineheight{1.25}\smash{\begin{tabular}[t]{l}\ss$\tilde d_n $\end{tabular}}}}%
    \put(0.15350104,0.0572984){\color[rgb]{0,0,0}\makebox(0,0)[lt]{\lineheight{1.25}\smash{\begin{tabular}[t]{l}\ss$d_0$\end{tabular}}}}%
    \put(0.77113966,0.06083275){\color[rgb]{0,0,0}\makebox(0,0)[lt]{\lineheight{1.25}\smash{\begin{tabular}[t]{l}\ss$\tilde d_0$\end{tabular}}}}%
    \put(0,0){\includegraphics[width=\unitlength,page=5]{arm.pdf}}%
    \put(0.07839141,0.27869517){\color[rgb]{0,0,0}\makebox(0,0)[lt]{\lineheight{1.25}\smash{\begin{tabular}[t]{l}\ss$\alpha_\nmo$\end{tabular}}}}%
    \put(0.10379251,0.22288886){\color[rgb]{0,0,0}\makebox(0,0)[lt]{\lineheight{1.25}\smash{\begin{tabular}[t]{l}\ss$d_\nmo$\end{tabular}}}}%
    \put(0.80262626,0.28385953){\color[rgb]{0,0,0}\makebox(0,0)[lt]{\lineheight{1.25}\smash{\begin{tabular}[t]{l}\ss$\tilde d_\nmo$\end{tabular}}}}%
    \put(0,0){\includegraphics[width=\unitlength,page=6]{arm.pdf}}%
    \put(0.83914798,0.34155474){\color[rgb]{0,0,0}\rotatebox{8.0467433}{\makebox(0,0)[lt]{\lineheight{1.25}\smash{\begin{tabular}[t]{l}        \end{tabular}}}}}%
    \put(0,0){\includegraphics[width=\unitlength,page=7]{arm.pdf}}%
    \put(0.78764157,0.32294715){\color[rgb]{0,0,0}\makebox(0,0)[lt]{\lineheight{1.25}\smash{\begin{tabular}[t]{l}\ss$\tilde\alpha_\nmo$\end{tabular}}}}%
    \put(0.22527471,0.0957692){\color[rgb]{0,0,0}\makebox(0,0)[lt]{\lineheight{1.25}\smash{\begin{tabular}[t]{l}\ss$d_1$\end{tabular}}}}%
    \put(0,0){\includegraphics[width=\unitlength,page=8]{arm.pdf}}%
    \put(0.84460359,0.10445183){\color[rgb]{0,0,0}\makebox(0,0)[lt]{\lineheight{1.25}\smash{\begin{tabular}[t]{l}\ss$\tilde d_1$\end{tabular}}}}%
  \end{picture}%
\endgroup%

%% file: figures/klein1.pdf_tex
\begingroup%
  \makeatletter%
  \providecommand\color[2][]{%
    \errmessage{(Inkscape) Color is used for the text in Inkscape, but the package 'color.sty' is not loaded}%
    \renewcommand\color[2][]{}%
  }%
  \providecommand\transparent[1]{%
    \errmessage{(Inkscape) Transparency is used (non-zero) for the text in Inkscape, but the package 'transparent.sty' is not loaded}%
    \renewcommand\transparent[1]{}%
  }%
  \providecommand\rotatebox[2]{#2}%
  \newcommand*\fsize{\dimexpr\f@size pt\relax}%
  \newcommand*\lineheight[1]{\fontsize{\fsize}{#1\fsize}\selectfont}%
  \ifx\svgwidth\undefined%
    \setlength{\unitlength}{283.00928311bp}%
    \ifx\svgscale\undefined%
      \relax%
    \else%
      \setlength{\unitlength}{\unitlength * \real{\svgscale}}%
    \fi%
  \else%
    \setlength{\unitlength}{\svgwidth}%
  \fi%
  \global\let\svgwidth\undefined%
  \global\let\svgscale\undefined%
  \makeatother%
  \begin{picture}(1,0.33208837)%
    \lineheight{1}%
    \setlength\tabcolsep{0pt}%
    \put(0,0){\includegraphics[width=\unitlength,page=1]{klein1.pdf}}%
    \put(0.97845849,0.25540552){\color[rgb]{0.05882353,0.01960784,0.01960784}\makebox(0,0)[lt]{\lineheight{1.25}\smash{\begin{tabular}[t]{l}\ss$\tG(0)$\end{tabular}}}}%
    \put(0.90828719,0.09207881){\color[rgb]{0.03529412,0,0}\makebox(0,0)[lt]{\lineheight{1.25}\smash{\begin{tabular}[t]{l} \end{tabular}}}}%
    \put(0,0){\includegraphics[width=\unitlength,page=2]{klein1.pdf}}%
    \put(0.63107987,0.03937036){\color[rgb]{0.05882353,0.01960784,0.01960784}\makebox(0,0)[lt]{\lineheight{1.25}\smash{\begin{tabular}[t]{l}\s (b)\end{tabular}}}}%
    \put(0.85196642,0.30324022){\color[rgb]{0.05882353,0.01960784,0.01960784}\makebox(0,0)[lt]{\lineheight{1.25}\smash{\begin{tabular}[t]{l}\s$\eta$\end{tabular}}}}%
    \put(0,0){\includegraphics[width=\unitlength,page=3]{klein1.pdf}}%
    \put(0.27226561,0.09207881){\color[rgb]{0.03529412,0,0}\makebox(0,0)[lt]{\lineheight{1.25}\smash{\begin{tabular}[t]{l} \end{tabular}}}}%
    \put(-0.00003105,0.03598877){\color[rgb]{0.05882353,0.01960784,0.01960784}\makebox(0,0)[lt]{\lineheight{1.25}\smash{\begin{tabular}[t]{l}\s (a)\end{tabular}}}}%
    \put(0,0){\includegraphics[width=\unitlength,page=4]{klein1.pdf}}%
    \put(0.64239489,0.28081131){\color[rgb]{0.05882353,0.01960784,0.01960784}\makebox(0,0)[lt]{\lineheight{1.25}\smash{\begin{tabular}[t]{l}\ss$\tG(L)$\end{tabular}}}}%
  \end{picture}%
\endgroup%

%% file: figures/lemma.pdf_tex
\begingroup%
  \makeatletter%
  \providecommand\color[2][]{%
    \errmessage{(Inkscape) Color is used for the text in Inkscape, but the package 'color.sty' is not loaded}%
    \renewcommand\color[2][]{}%
  }%
  \providecommand\transparent[1]{%
    \errmessage{(Inkscape) Transparency is used (non-zero) for the text in Inkscape, but the package 'transparent.sty' is not loaded}%
    \renewcommand\transparent[1]{}%
  }%
  \providecommand\rotatebox[2]{#2}%
  \newcommand*\fsize{\dimexpr\f@size pt\relax}%
  \newcommand*\lineheight[1]{\fontsize{\fsize}{#1\fsize}\selectfont}%
  \ifx\svgwidth\undefined%
    \setlength{\unitlength}{711.56998281bp}%
    \ifx\svgscale\undefined%
      \relax%
    \else%
      \setlength{\unitlength}{\unitlength * \real{\svgscale}}%
    \fi%
  \else%
    \setlength{\unitlength}{\svgwidth}%
  \fi%
  \global\let\svgwidth\undefined%
  \global\let\svgscale\undefined%
  \makeatother%
  \begin{picture}(1,0.34617468)%
    \lineheight{1}%
    \setlength\tabcolsep{0pt}%
    \put(0,0){\includegraphics[width=\unitlength,page=1]{lemma.pdf}}%
    \put(0.228294,0.27083995){\color[rgb]{0.05882353,0.01960784,0.01960784}\makebox(0,0)[lt]{\lineheight{1.25}\smash{\begin{tabular}[t]{l}\ss$\tG$\end{tabular}}}}%
    \put(0.37406862,0.22162374){\color[rgb]{0.05882353,0.01960784,0.01960784}\makebox(0,0)[lt]{\lineheight{1.25}\smash{\begin{tabular}[t]{l}\ss$\tG(0)$\end{tabular}}}}%
    \put(-0.00001235,0.2217943){\color[rgb]{0.05882353,0.01960784,0.01960784}\makebox(0,0)[lt]{\lineheight{1.25}\smash{\begin{tabular}[t]{l}\ss$\tG(L)$\end{tabular}}}}%
    \put(0,0){\includegraphics[width=\unitlength,page=2]{lemma.pdf}}%
    \put(0.96247974,0.21897458){\color[rgb]{0.05882353,0.01960784,0.01960784}\makebox(0,0)[lt]{\lineheight{1.25}\smash{\begin{tabular}[t]{l}\ss$\tG(0)$\end{tabular}}}}%
    \put(0.59605137,0.21983992){\color[rgb]{0.05882353,0.01960784,0.01960784}\makebox(0,0)[lt]{\lineheight{1.25}\smash{\begin{tabular}[t]{l}\ss$\tG(L)$\end{tabular}}}}%
    \put(0.82217103,0.25275427){\color[rgb]{0.05882353,0.01960784,0.01960784}\makebox(0,0)[lt]{\lineheight{1.25}\smash{\begin{tabular}[t]{l}\ss $\tG$\end{tabular}}}}%
    \put(0.01834462,0.03293614){\color[rgb]{0.05882353,0.01960784,0.01960784}\makebox(0,0)[lt]{\lineheight{1.25}\smash{\begin{tabular}[t]{l}\s(a)\end{tabular}}}}%
    \put(0.62281615,0.03479788){\color[rgb]{0.05882353,0.01960784,0.01960784}\makebox(0,0)[lt]{\lineheight{1.25}\smash{\begin{tabular}[t]{l}\s(b)\end{tabular}}}}%
    \put(0.24897879,0.34012237){\color[rgb]{0.05882353,0.01960784,0.01960784}\makebox(0,0)[lt]{\lineheight{1.25}\smash{\begin{tabular}[t]{l}\s$\eta$\end{tabular}}}}%
    \put(0.91365171,0.29894166){\color[rgb]{0.05882353,0.01960784,0.01960784}\makebox(0,0)[lt]{\lineheight{1.25}\smash{\begin{tabular}[t]{l}\s$\eta$\end{tabular}}}}%
    \put(0,0){\includegraphics[width=\unitlength,page=3]{lemma.pdf}}%
    \put(0.73455572,0.00064229){\color[rgb]{0.05882353,0.01960784,0.01960784}\makebox(0,0)[lt]{\lineheight{1.25}\smash{\begin{tabular}[t]{l}\s$S$\end{tabular}}}}%
    \put(0.79285957,0.14214185){\color[rgb]{0.05882353,0.01960784,0.01960784}\makebox(0,0)[lt]{\lineheight{1.25}\smash{\begin{tabular}[t]{l}\ss$\tG(t_*)$\end{tabular}}}}%
    \put(0,0){\includegraphics[width=\unitlength,page=4]{lemma.pdf}}%
    \put(0.85486011,0.18270632){\color[rgb]{0.05882353,0.01960784,0.01960784}\makebox(0,0)[lt]{\lineheight{1.25}\smash{\begin{tabular}[t]{l}\s $\Omega$\end{tabular}}}}%
    \put(0,0){\includegraphics[width=\unitlength,page=5]{lemma.pdf}}%
  \end{picture}%
\endgroup%

%% file: figures/klein2.pdf_tex
\begingroup%
  \makeatletter%
  \providecommand\color[2][]{%
    \errmessage{(Inkscape) Color is used for the text in Inkscape, but the package 'color.sty' is not loaded}%
    \renewcommand\color[2][]{}%
  }%
  \providecommand\transparent[1]{%
    \errmessage{(Inkscape) Transparency is used (non-zero) for the text in Inkscape, but the package 'transparent.sty' is not loaded}%
    \renewcommand\transparent[1]{}%
  }%
  \providecommand\rotatebox[2]{#2}%
  \newcommand*\fsize{\dimexpr\f@size pt\relax}%
  \newcommand*\lineheight[1]{\fontsize{\fsize}{#1\fsize}\selectfont}%
  \ifx\svgwidth\undefined%
    \setlength{\unitlength}{132.96472865bp}%
    \ifx\svgscale\undefined%
      \relax%
    \else%
      \setlength{\unitlength}{\unitlength * \real{\svgscale}}%
    \fi%
  \else%
    \setlength{\unitlength}{\svgwidth}%
  \fi%
  \global\let\svgwidth\undefined%
  \global\let\svgscale\undefined%
  \makeatother%
  \begin{picture}(1,0.36777025)%
    \lineheight{1}%
    \setlength\tabcolsep{0pt}%
    \put(0,0){\includegraphics[width=\unitlength,page=1]{klein2.pdf}}%
    \put(0.19769673,0.26046487){\color[rgb]{0.05882353,0.01960784,0.01960784}\transparent{0.21265402}\rotatebox{0.696949}{\makebox(0,0)[lt]{\lineheight{1.25}\smash{\begin{tabular}[t]{l}  \end{tabular}}}}}%
    \put(0.09183619,0.20501282){\color[rgb]{0.05882353,0.01960784,0.01960784}\rotatebox{0.696949}{\makebox(0,0)[lt]{\lineheight{1.25}\smash{\begin{tabular}[t]{l}\ss$\tG(nL)$\end{tabular}}}}}%
    \put(0.8827086,0.23446193){\color[rgb]{0.05882353,0.01960784,0.01960784}\rotatebox{0.696949}{\makebox(0,0)[lt]{\lineheight{1.25}\smash{\begin{tabular}[t]{l}\ss$\tG(-nL)$\end{tabular}}}}}%
    \put(0.54789618,0.11988796){\color[rgb]{0.05882353,0.01960784,0.01960784}\rotatebox{0.696949}{\makebox(0,0)[lt]{\lineheight{1.25}\smash{\begin{tabular}[t]{l}\ss$g$\end{tabular}}}}}%
  \end{picture}%
\endgroup%

%% file: figures/piece.pdf_tex
\begingroup%
  \makeatletter%
  \providecommand\color[2][]{%
    \errmessage{(Inkscape) Color is used for the text in Inkscape, but the package 'color.sty' is not loaded}%
    \renewcommand\color[2][]{}%
  }%
  \providecommand\transparent[1]{%
    \errmessage{(Inkscape) Transparency is used (non-zero) for the text in Inkscape, but the package 'transparent.sty' is not loaded}%
    \renewcommand\transparent[1]{}%
  }%
  \providecommand\rotatebox[2]{#2}%
  \newcommand*\fsize{\dimexpr\f@size pt\relax}%
  \newcommand*\lineheight[1]{\fontsize{\fsize}{#1\fsize}\selectfont}%
  \ifx\svgwidth\undefined%
    \setlength{\unitlength}{387.66430952bp}%
    \ifx\svgscale\undefined%
      \relax%
    \else%
      \setlength{\unitlength}{\unitlength * \real{\svgscale}}%
    \fi%
  \else%
    \setlength{\unitlength}{\svgwidth}%
  \fi%
  \global\let\svgwidth\undefined%
  \global\let\svgscale\undefined%
  \makeatother%
  \begin{picture}(1,0.43024044)%
    \lineheight{1}%
    \setlength\tabcolsep{0pt}%
    \put(0,0){\includegraphics[width=\unitlength,page=1]{piece.pdf}}%
    \put(0.66288572,0.29138555){\color[rgb]{0,0,0}\makebox(0,0)[lt]{\lineheight{1.25}\smash{\begin{tabular}[t]{l}\s$\g$\end{tabular}}}}%
    \put(0,0){\includegraphics[width=\unitlength,page=2]{piece.pdf}}%
    \put(0.37094089,0.40802206){\color[rgb]{0,0,0}\makebox(0,0)[lt]{\lineheight{1.25}\smash{\begin{tabular}[t]{l}\s$\G(t_*)$\end{tabular}}}}%
    \put(0,0){\includegraphics[width=\unitlength,page=3]{piece.pdf}}%
    \put(0.93447844,0.13066124){\color[rgb]{0,0,0}\makebox(0,0)[lt]{\lineheight{1.25}\smash{\begin{tabular}[t]{l}\s$x$\end{tabular}}}}%
    \put(0.79146483,0.00662018){\color[rgb]{0,0,0}\makebox(0,0)[lt]{\lineheight{1.25}\smash{\begin{tabular}[t]{l}\s$X(b)$\end{tabular}}}}%
    \put(0.01312214,0.010085){\color[rgb]{0,0,0}\makebox(0,0)[lt]{\lineheight{1.25}\smash{\begin{tabular}[t]{l}\s$X(a)$\end{tabular}}}}%
    \put(0,0){\includegraphics[width=\unitlength,page=4]{piece.pdf}}%
    \put(0.39472133,0.01530532){\color[rgb]{0,0,0}\makebox(0,0)[lt]{\lineheight{1.25}\smash{\begin{tabular}[t]{l}\s$d$\end{tabular}}}}%
    \put(0,0){\includegraphics[width=\unitlength,page=5]{piece.pdf}}%
  \end{picture}%
\endgroup%

%% file: figures/approx.pdf_tex
\begingroup%
  \makeatletter%
  \providecommand\color[2][]{%
    \errmessage{(Inkscape) Color is used for the text in Inkscape, but the package 'color.sty' is not loaded}%
    \renewcommand\color[2][]{}%
  }%
  \providecommand\transparent[1]{%
    \errmessage{(Inkscape) Transparency is used (non-zero) for the text in Inkscape, but the package 'transparent.sty' is not loaded}%
    \renewcommand\transparent[1]{}%
  }%
  \providecommand\rotatebox[2]{#2}%
  \newcommand*\fsize{\dimexpr\f@size pt\relax}%
  \newcommand*\lineheight[1]{\fontsize{\fsize}{#1\fsize}\selectfont}%
  \ifx\svgwidth\undefined%
    \setlength{\unitlength}{18.92409482bp}%
    \ifx\svgscale\undefined%
      \relax%
    \else%
      \setlength{\unitlength}{\unitlength * \real{\svgscale}}%
    \fi%
  \else%
    \setlength{\unitlength}{\svgwidth}%
  \fi%
  \global\let\svgwidth\undefined%
  \global\let\svgscale\undefined%
  \makeatother%
  \begin{picture}(1,0.83816707)%
    \lineheight{1}%
    \setlength\tabcolsep{0pt}%
    \put(0.98218544,0.52930393){\color[rgb]{0,0,0}\makebox(0,0)[lt]{\lineheight{1.25}\smash{\begin{tabular}[t]{l} \end{tabular}}}}%
    \put(0,0){\includegraphics[width=\unitlength,page=1]{approx.pdf}}%
    \put(0.69503434,0.4435441){\color[rgb]{0,0,0}\makebox(0,0)[lt]{\lineheight{1.25}\smash{\begin{tabular}[t]{l}\ss$\epsilon$\end{tabular}}}}%
    \put(0.72326659,0.33484008){\color[rgb]{0,0,0}\makebox(0,0)[lt]{\lineheight{1.25}\smash{\begin{tabular}[t]{l}\ss$L$\end{tabular}}}}%
    \put(0.93004906,0.4727703){\color[rgb]{0,0,0}\makebox(0,0)[lt]{\lineheight{1.25}\smash{\begin{tabular}[t]{l}\ss$L\epsilon$\end{tabular}}}}%
    \put(0,0){\includegraphics[width=\unitlength,page=2]{approx.pdf}}%
    \put(0.53000233,0.46890654){\color[rgb]{0,0,0}\rotatebox{19.135001}{\makebox(0,0)[lt]{\lineheight{1.25}\smash{\begin{tabular}[t]{l}\ss$L\cos\epsilon$\end{tabular}}}}}%
  \end{picture}%
\endgroup%

%% file: figures/contact.pdf_tex
\begingroup%
  \makeatletter%
  \providecommand\color[2][]{%
    \errmessage{(Inkscape) Color is used for the text in Inkscape, but the package 'color.sty' is not loaded}%
    \renewcommand\color[2][]{}%
  }%
  \providecommand\transparent[1]{%
    \errmessage{(Inkscape) Transparency is used (non-zero) for the text in Inkscape, but the package 'transparent.sty' is not loaded}%
    \renewcommand\transparent[1]{}%
  }%
  \providecommand\rotatebox[2]{#2}%
  \newcommand*\fsize{\dimexpr\f@size pt\relax}%
  \newcommand*\lineheight[1]{\fontsize{\fsize}{#1\fsize}\selectfont}%
  \ifx\svgwidth\undefined%
    \setlength{\unitlength}{210.85302638bp}%
    \ifx\svgscale\undefined%
      \relax%
    \else%
      \setlength{\unitlength}{\unitlength * \real{\svgscale}}%
    \fi%
  \else%
    \setlength{\unitlength}{\svgwidth}%
  \fi%
  \global\let\svgwidth\undefined%
  \global\let\svgscale\undefined%
  \makeatother%
  \begin{picture}(1,0.76434717)%
    \lineheight{1}%
    \setlength\tabcolsep{0pt}%
    \put(0,0){\includegraphics[width=\unitlength,page=1]{contact.pdf}}%
    \put(0.53473229,0.04984074){\color[rgb]{0,0,0}\makebox(0,0)[lt]{\lineheight{1.25}\smash{\begin{tabular}[t]{l}\s$\G$\end{tabular}}}}%
    \put(0,0){\includegraphics[width=\unitlength,page=2]{contact.pdf}}%
    \put(0.86372333,0.68684002){\color[rgb]{0,0,0}\makebox(0,0)[lt]{\lineheight{1.25}\smash{\begin{tabular}[t]{l}\s$\g$\end{tabular}}}}%
    \put(0.34955304,0.27809826){\color[rgb]{0,0,0}\makebox(0,0)[lt]{\lineheight{1.25}\smash{\begin{tabular}[t]{l}\s$\theta$\end{tabular}}}}%
    \put(0.65019417,0.55370589){\color[rgb]{0,0,0}\makebox(0,0)[lt]{\lineheight{1.25}\smash{\begin{tabular}[t]{l}\ss$(x,y)$\end{tabular}}}}%
    \put(0,0){\includegraphics[width=\unitlength,page=3]{contact.pdf}}%
    \put(0.02419329,0.08990087){\color[rgb]{0,0,0}\makebox(0,0)[lt]{\lineheight{1.25}\smash{\begin{tabular}[t]{l}\ss$(X,Y)$\end{tabular}}}}%
    \put(0.35254759,0.48124023){\color[rgb]{0,0,0}\makebox(0,0)[lt]{\lineheight{1.25}\smash{\begin{tabular}[t]{l}\s$\ell$\end{tabular}}}}%
    \put(0,0){\includegraphics[width=\unitlength,page=4]{contact.pdf}}%
  \end{picture}%
\endgroup%

%% file: figures/maslov1.pdf_tex
\begingroup%
  \makeatletter%
  \providecommand\color[2][]{%
    \errmessage{(Inkscape) Color is used for the text in Inkscape, but the package 'color.sty' is not loaded}%
    \renewcommand\color[2][]{}%
  }%
  \providecommand\transparent[1]{%
    \errmessage{(Inkscape) Transparency is used (non-zero) for the text in Inkscape, but the package 'transparent.sty' is not loaded}%
    \renewcommand\transparent[1]{}%
  }%
  \providecommand\rotatebox[2]{#2}%
  \newcommand*\fsize{\dimexpr\f@size pt\relax}%
  \newcommand*\lineheight[1]{\fontsize{\fsize}{#1\fsize}\selectfont}%
  \ifx\svgwidth\undefined%
    \setlength{\unitlength}{93.3674984bp}%
    \ifx\svgscale\undefined%
      \relax%
    \else%
      \setlength{\unitlength}{\unitlength * \real{\svgscale}}%
    \fi%
  \else%
    \setlength{\unitlength}{\svgwidth}%
  \fi%
  \global\let\svgwidth\undefined%
  \global\let\svgscale\undefined%
  \makeatother%
  \begin{picture}(1,0.40093382)%
    \lineheight{1}%
    \setlength\tabcolsep{0pt}%
    \put(0,0){\includegraphics[width=\unitlength,page=1]{maslov1.pdf}}%
    \put(0.61267012,0.01703056){\color[rgb]{0.03137255,0.03137255,0.03137255}\makebox(0,0)[lt]{\lineheight{1.25}\smash{\begin{tabular}[t]{l}\s$\mu=-1$\end{tabular}}}}%
    \put(0,0){\includegraphics[width=\unitlength,page=2]{maslov1.pdf}}%
    \put(0.06424271,0.01168012){\color[rgb]{0.03137255,0.03137255,0.03137255}\makebox(0,0)[lt]{\lineheight{1.25}\smash{\begin{tabular}[t]{l}\s$\mu=+1$\end{tabular}}}}%
    \put(1.02410011,0.19740051){\color[rgb]{0.03137255,0.03137255,0.03137255}\makebox(0,0)[lt]{\lineheight{1.25}\smash{\begin{tabular}[t]{l} \end{tabular}}}}%
  \end{picture}%
\endgroup%

%% file: figures/waves1.pdf_tex
\begingroup%
  \makeatletter%
  \providecommand\color[2][]{%
    \errmessage{(Inkscape) Color is used for the text in Inkscape, but the package 'color.sty' is not loaded}%
    \renewcommand\color[2][]{}%
  }%
  \providecommand\transparent[1]{%
    \errmessage{(Inkscape) Transparency is used (non-zero) for the text in Inkscape, but the package 'transparent.sty' is not loaded}%
    \renewcommand\transparent[1]{}%
  }%
  \providecommand\rotatebox[2]{#2}%
  \newcommand*\fsize{\dimexpr\f@size pt\relax}%
  \newcommand*\lineheight[1]{\fontsize{\fsize}{#1\fsize}\selectfont}%
  \ifx\svgwidth\undefined%
    \setlength{\unitlength}{298.37528229bp}%
    \ifx\svgscale\undefined%
      \relax%
    \else%
      \setlength{\unitlength}{\unitlength * \real{\svgscale}}%
    \fi%
  \else%
    \setlength{\unitlength}{\svgwidth}%
  \fi%
  \global\let\svgwidth\undefined%
  \global\let\svgscale\undefined%
  \makeatother%
  \begin{picture}(1,0.1437552)%
    \lineheight{1}%
    \setlength\tabcolsep{0pt}%
    \put(0,0){\includegraphics[width=\unitlength,page=1]{waves1.pdf}}%
  \end{picture}%
\endgroup%

%% file: figures/cusps.pdf_tex
\begingroup%
  \makeatletter%
  \providecommand\color[2][]{%
    \errmessage{(Inkscape) Color is used for the text in Inkscape, but the package 'color.sty' is not loaded}%
    \renewcommand\color[2][]{}%
  }%
  \providecommand\transparent[1]{%
    \errmessage{(Inkscape) Transparency is used (non-zero) for the text in Inkscape, but the package 'transparent.sty' is not loaded}%
    \renewcommand\transparent[1]{}%
  }%
  \providecommand\rotatebox[2]{#2}%
  \newcommand*\fsize{\dimexpr\f@size pt\relax}%
  \newcommand*\lineheight[1]{\fontsize{\fsize}{#1\fsize}\selectfont}%
  \ifx\svgwidth\undefined%
    \setlength{\unitlength}{331.53710175bp}%
    \ifx\svgscale\undefined%
      \relax%
    \else%
      \setlength{\unitlength}{\unitlength * \real{\svgscale}}%
    \fi%
  \else%
    \setlength{\unitlength}{\svgwidth}%
  \fi%
  \global\let\svgwidth\undefined%
  \global\let\svgscale\undefined%
  \makeatother%
  \begin{picture}(1,0.34467687)%
    \lineheight{1}%
    \setlength\tabcolsep{0pt}%
    \put(0,0){\includegraphics[width=\unitlength,page=1]{cusps.pdf}}%
    \put(0.24932963,0.21830484){\makebox(0,0)[lt]{\lineheight{1.25}\smash{\begin{tabular}[t]{l}\ss$(1,{1\over 2}, 1)$\end{tabular}}}}%
    \put(0.25030649,0.03707177){\makebox(0,0)[lt]{\lineheight{1.25}\smash{\begin{tabular}[t]{l}\ss$(0,{1\over 2},-1)$\end{tabular}}}}%
    \put(0.8827429,0.03732962){\makebox(0,0)[lt]{\lineheight{1.25}\smash{\begin{tabular}[t]{l}\ss$(0,-{1\over 2},1)$\end{tabular}}}}%
    \put(0.8827429,0.21830484){\makebox(0,0)[lt]{\lineheight{1.25}\smash{\begin{tabular}[t]{l}\ss$(-1,-{1\over 2},1)$\end{tabular}}}}%
  \end{picture}%
\endgroup%

%% file: figures/fish.pdf_tex
\begingroup%
  \makeatletter%
  \providecommand\color[2][]{%
    \errmessage{(Inkscape) Color is used for the text in Inkscape, but the package 'color.sty' is not loaded}%
    \renewcommand\color[2][]{}%
  }%
  \providecommand\transparent[1]{%
    \errmessage{(Inkscape) Transparency is used (non-zero) for the text in Inkscape, but the package 'transparent.sty' is not loaded}%
    \renewcommand\transparent[1]{}%
  }%
  \providecommand\rotatebox[2]{#2}%
  \newcommand*\fsize{\dimexpr\f@size pt\relax}%
  \newcommand*\lineheight[1]{\fontsize{\fsize}{#1\fsize}\selectfont}%
  \ifx\svgwidth\undefined%
    \setlength{\unitlength}{364.31298065bp}%
    \ifx\svgscale\undefined%
      \relax%
    \else%
      \setlength{\unitlength}{\unitlength * \real{\svgscale}}%
    \fi%
  \else%
    \setlength{\unitlength}{\svgwidth}%
  \fi%
  \global\let\svgwidth\undefined%
  \global\let\svgscale\undefined%
  \makeatother%
  \begin{picture}(1,0.13703053)%
    \lineheight{1}%
    \setlength\tabcolsep{0pt}%
    \put(0,0){\includegraphics[width=\unitlength,page=1]{fish.pdf}}%
    \put(0.85819195,0.10192854){\color[rgb]{0.03137255,0.03137255,0.03137255}\makebox(0,0)[lt]{\lineheight{1.25}\smash{\begin{tabular}[t]{l}\s$\g$\end{tabular}}}}%
    \put(0.30323419,0.11795925){\color[rgb]{0.03137255,0.03137255,0.03137255}\makebox(0,0)[lt]{\lineheight{1.25}\smash{\begin{tabular}[t]{l}\s$\G$\end{tabular}}}}%
  \end{picture}%
\endgroup%

%% file: figures/fish4.pdf_tex
\begingroup%
  \makeatletter%
  \providecommand\color[2][]{%
    \errmessage{(Inkscape) Color is used for the text in Inkscape, but the package 'color.sty' is not loaded}%
    \renewcommand\color[2][]{}%
  }%
  \providecommand\transparent[1]{%
    \errmessage{(Inkscape) Transparency is used (non-zero) for the text in Inkscape, but the package 'transparent.sty' is not loaded}%
    \renewcommand\transparent[1]{}%
  }%
  \providecommand\rotatebox[2]{#2}%
  \newcommand*\fsize{\dimexpr\f@size pt\relax}%
  \newcommand*\lineheight[1]{\fontsize{\fsize}{#1\fsize}\selectfont}%
  \ifx\svgwidth\undefined%
    \setlength{\unitlength}{911.14263916bp}%
    \ifx\svgscale\undefined%
      \relax%
    \else%
      \setlength{\unitlength}{\unitlength * \real{\svgscale}}%
    \fi%
  \else%
    \setlength{\unitlength}{\svgwidth}%
  \fi%
  \global\let\svgwidth\undefined%
  \global\let\svgscale\undefined%
  \makeatother%
  \begin{picture}(1,0.69757367)%
    \lineheight{1}%
    \setlength\tabcolsep{0pt}%
    \put(0,0){\includegraphics[width=\unitlength,page=1]{fish4.pdf}}%
    \put(0.60112247,0.44008959){\color[rgb]{0.00392157,0.00392157,0.00392157}\makebox(0,0)[lt]{\lineheight{1.25}\smash{\begin{tabular}[t]{l}\ss$\ell=0.9272$ (parabolic), $\rho(\g)=1,\ \mu(\g)=0$ \end{tabular}}}}%
    \put(0,0){\includegraphics[width=\unitlength,page=2]{fish4.pdf}}%
    \put(0.04105082,0.43562581){\color[rgb]{0.00392157,0.00392157,0.00392157}\makebox(0,0)[lt]{\lineheight{1.25}\smash{\begin{tabular}[t]{l}\ss$\ell=0.4$ (hyperbolic), $\rho(\g)=1,\ \mu(\g)=0.$\end{tabular}}}}%
    \put(0,0){\includegraphics[width=\unitlength,page=3]{fish4.pdf}}%
    \put(0.04739005,0.10880912){\color[rgb]{0.00392157,0.00392157,0.00392157}\makebox(0,0)[lt]{\lineheight{1.25}\smash{\begin{tabular}[t]{l}\ss$\ell=1$ (parabolic), $\rho(\g)=0,\ \mu(\g)=2.$\end{tabular}}}}%
    \put(0,0){\includegraphics[width=\unitlength,page=4]{fish4.pdf}}%
    \put(0.6079499,0.10765539){\color[rgb]{0.00392157,0.00392157,0.00392157}\makebox(0,0)[lt]{\lineheight{1.25}\smash{\begin{tabular}[t]{l}\ss$\ell=2$ (hyperbolic),\end{tabular}}}}%
    \put(0.6079499,0.08182609){\color[rgb]{0.00392157,0.00392157,0.00392157}\makebox(0,0)[lt]{\lineheight{1.25}\smash{\begin{tabular}[t]{l}\ss$\rho(\g)=0,\ \mu(\g)=2.$ \end{tabular}}}}%
    \put(0,0){\includegraphics[width=\unitlength,page=5]{fish4.pdf}}%
  \end{picture}%
\endgroup%